\documentclass[preprint]{imsart}
\RequirePackage[OT1]{fontenc} \RequirePackage{amsthm,amsmath}
\RequirePackage[numbers]{natbib}
\RequirePackage[colorlinks,citecolor=blue,urlcolor=blue]{hyperref}
\usepackage{amsmath}
\usepackage{amssymb}
\usepackage{amsmath}
\usepackage{amsfonts}
\usepackage[american]{babel}
\usepackage{graphicx}
\usepackage{flafter}
\usepackage[section]{placeins}
\startlocaldefs

\def\E{{\mathbb E }}

\def\goin{\to\infty}
\def\Bbb E{\mathbb{E}}
\def\Bbb R{\mathbb{R}}
\newtheorem{proposition}{Proposition}
\newtheorem{lemma}{Lemma}
\newtheorem{theorem}{Theorem}
\newtheorem{corollary}{Corollary}
\newtheorem{remark}{Remark}
\makeatletter \@addtoreset{equation}{section}

\makeatother

\font\tencmmib=cmmib10 \skewchar\tencmmib '60
\newfam\cmmibfam
\textfont\cmmibfam=\tencmmib

\font\tenmsb=msbm10 
\def\Bbb#1{\hbox{\tenmsb#1}}

\def\bbox{\quad\hbox{\vrule \vbox{\hrule \vskip2pt \hbox{\hskip2pt
\vbox{\hsize=1pt}\hskip2pt} \vskip2pt\hrule}\vrule}}
\def\lessim{\ \lower4pt\hbox{$
\buildrel{\displaystyle <}\over\sim$}\ }
\def\gessim{\ \lower4pt\hbox{$\buildrel{\displaystyle >}
\over\sim$}\ }

\def\goin{\to \infty}
\def\go0{\to 0}

\def\leftitem#1{\item{\hbox to\parindent{\enspace#1\hfill}}}

\def\qed{{$\hfill \bbox$}}

\def\sg{\sigma}

\def\sg2{\sigma^2}

\def\__{_{\infty}}
\numberwithin{equation}{section} \theoremstyle{plain}

\newcommand{\1}{{\rm 1}\kern-0.24em{\rm I}}
\def\E{\mathbb E}
\def\R{\mathbb R}
\newtheorem{assumption}{Assumption}

\endlocaldefs

\begin{document}

\begin{frontmatter}
\title{New asymptotic results in principal component analysis}\runtitle{Asymptotics in principal component analysis}

\begin{aug}
\author{\fnms{Vladimir} \snm{Koltchinskii}\thanksref{t1}\ead[label=e1]{vlad@math.gatech.edu}} and 
\author{\fnms{Karim} \snm{Lounici}\thanksref{m1}\ead[label=e2]{klounici@math.gatech.edu}}
%\thankstext{t1}{Supported in part by NSF grants
%DMS-09-06880 and CCF-0808863}
\thankstext{t1}{Supported in part by NSF Grants DMS-1509739, DMS-1207808, CCF-1523768 and CCF-1415498}
\thankstext{m1}{Supported in part by Simons Grant 315477 and NSF Career Grant, DMS-1454515}
\runauthor{V. Koltchinskii and K. Lounici}

\affiliation{Georgia Institute of Technology\thanksmark{m1}}

\address{School of Mathematics\\
Georgia Institute of Technology\\
Atlanta, GA 30332-0160\\
\printead{e1}\\
% \phantom{E-mail:\ }
\printead*{e2}
}
\end{aug}

\begin{abstract}
Let $X$ be a mean zero Gaussian random vector in a separable Hilbert space 
${\mathbb H}$ with covariance operator $\Sigma:={\mathbb E}(X\otimes X).$
Let $\Sigma=\sum_{r\geq 1}\mu_r P_r$ be the spectral decomposition 
of $\Sigma$ with distinct eigenvalues $\mu_1>\mu_2> \dots$ and 
the corresponding spectral projectors $P_1, P_2, \dots.$ 
Given a sample $X_1,\dots, X_n$ of size $n$ of i.i.d. copies of $X,$ the sample 
covariance operator is defined as $\hat \Sigma_n := n^{-1}\sum_{j=1}^n X_j\otimes X_j.$
The main goal of principal component analysis is to estimate spectral projectors $P_1, P_2, \dots$
by their empirical counterparts $\hat P_1, \hat P_2, \dots$ properly defined in terms of 
spectral decomposition of the sample covariance operator $\hat \Sigma_n.$
The aim of this paper is to study asymptotic distributions of important statistics 
related to this problem, in particular, of statistic $\|\hat P_r-P_r\|_2^2,$ where $\|\cdot\|_2^2$
is the squared Hilbert--Schmidt norm. This is done in a ``high-complexity" asymptotic 
framework in which the so called effective rank ${\bf r}(\Sigma):=\frac{{\rm tr}(\Sigma)}{\|\Sigma\|_{\infty}}$
(${\rm tr}(\cdot)$ being the trace and $\|\cdot\|_{\infty}$ being the operator norm) of the true covariance 
$\Sigma$ is becoming large simultaneously with the sample size $n,$ but ${\bf r}(\Sigma)=o(n)$
as $n\to\infty.$  In this setting, we prove that, in the case of one-dimensional spectral projector
$P_r,$ the properly centered and normalized statistic $\|\hat P_r-P_r\|_2^2$ with {\it data-dependent}
centering and normalization converges in distribution to a Cauchy type limit. 
The proofs of this and other related results rely on perturbation analysis and Gaussian concentration.    
 \end{abstract}



\end{frontmatter}

\section{Introduction}

Let $X,X_1,\dots, X_n, \dots$ be i.i.d. random variables sampled from a Gaussian distribution in 
a separable Hilbert space ${\mathbb H}$ with zero mean and covariance operator $\Sigma:={\mathbb E}X\otimes X$
and let $\hat \Sigma=\hat \Sigma_n := n^{-1}\sum_{j=1}^n X_j\otimes X_j$ denote the sample 
covariance operator based on $(X_1,\dots, X_n).$\footnote{Given $u,v\in {\mathbb H},$
the tensor product $u\otimes v$ is a rank one linear operator defined as $(u\otimes v)x=u\langle v,x\rangle, x\in {\mathbb H}.$}  We will be interested in asymptotic properties 
of several statistics related to spectral projectors of sample covariance $\hat \Sigma$ (empirical 
spectral projectors) that could be potentially useful in principal component analysis (PCA) and its infinite dimensional versions such as functional PCA (see, e.g., \cite{Ramsay}) or kernel PCA in machine learning (see, e.g., \cite{Scholkopf}, \cite{Blanchard}). 

In the classical setting of a finite-dimensional space ${\mathbb H}={\mathbb R}^p$ of a fixed dimension $p,$ the large sample asymptotics of spectral characteristics of sample covariance were studied by 
Anderson  \cite{Anderson} who derived the joint asymptotic distribution of the sample eigenvalues and the associated sample eigenvectors (see also Theorem 13.5.1 in \cite{Andersonbook}). Later on,
similar results were established in the infinite-dimensional case (see, e.g., \cite{DPR}). Such an 
extension is rather straightforward provided that the ``complexity of the problem" characterized 
by such parameters as the trace ${\rm tr}(\Sigma)$ of the covariance operator $\Sigma$ remains 
fixed when the sample size $n$ tends to infinity. 

In the high-dimensional setting, when the dimension $p$ of the space grows simultaneously 
with the sample size $n,$ the problem has been primarily studied for so called 
{\it spiked covariance models} introduced by Johnstone and co-authors (see, e.g., \cite{Johnstone_Lu}). In this case, the covariance $\Sigma$ has a special 
structure, namely,  
$$\Sigma= \sum_{j=1}^m \lambda_j^2 (\theta_j\otimes \theta_j)+ \sigma^2 I_p,$$ 
where $m<p,$ $\theta_1, \dots, \theta_m$ are orthonormal vectors (``principal components"),
$\lambda_1^2>\dots >\lambda_m^2>0,$ $\sigma^2>0$ and $I_p$ is the $p\times p$
identity matrix. This means that the observation $X$ can be represented as
\footnote{assuming that the orthonormal vectors $\theta_1,\dots, \theta_m$ are extended to an 
orthonormal basis $\theta_1,\dots, \theta_p$ of ${\mathbb R}^p$} 
$$
X=\sum_{j=1}^m \lambda_j \xi_j \theta_j + \sigma \sum_{j=1}^p \eta_j \theta_j,
$$ 
where $\xi_j, \eta_j, j\geq 1$ are i.i.d. standard normal random variables.
Thus, $X$ can be viewed as an observation of a ``signal"  $\sum_{j=1}^m \lambda_j \xi_j \theta_j,$
consisting of $m$ ``spikes",  in an independent Gaussian white noise.
For such models, an elegant asymptotic theory has been developed 
based on the achievements of random matrix theory (see, e.g., the results of Paul \cite{Paul_2007} 
on asymptotics of eigenvectors of sample covariance in spiked covariance models and references therein).
The most interesting results were obtained in the case when $\frac{p}{n}\to c$ for 
some constant $c\in (0,+\infty).$ In this case, however, the eigenvectors of the sample covariance 
$\hat \Sigma_n$ fail to be consistent estimators of the eigenvectors of the true covariance
$\Sigma$ (see Johnstone and Lu \cite{Johnstone_Lu}) and this difficulty could not be 
overcome without further assumptions on the true eigenvectors such as, for instance,
their sparsity. This led to the development of various approaches to ``sparse PCA" 
(see, e.g., \cite{Johnstone_Lu_arxiv,Lounici2013,Ma2013,Paul_Johnstone2012,VuLei2012, Birnbaumetal} and references therein).

In this paper, we follow a somewhat different path.
It is well known that to ensure 
consistency of empirical spectral projectors as statistical estimators of spectral projectors 
of the true covariance $\Sigma$ one has to establish convergence of $\hat \Sigma$ to $\Sigma$
in the operator norm. In what follows, $\|\cdot\|_{\infty}$ will denote the operator norm (for bounded 
operators in ${\mathbb H}$), $\|\cdot\|_2$ will denote the Hilbert--Schmidt norm and $\|\cdot \|_1$
will denote the nuclear norm.  
We also 
use the notation ${\rm tr}(\Sigma)$ for the trace of $\Sigma$ and set 
$${\bf r}(\Sigma):=\frac{{\rm tr}(\Sigma)}{\|\Sigma\|_{\infty}}.$$ 
The last quantity is always dominated by the rank of operator $\Sigma$ and 
it is sometimes referred to as its {\it effective rank}. It was pointed out by Vershynin \cite{Vershynin}
that the effective rank could be used to provide non-asymptotic upper bounds on the size 
of the operator norm $\|\hat \Sigma-\Sigma\|_{\infty}$ with rather weak (logarithmic) dependence
on the dimension and this approach was later used in statistical literature (see \cite{BuneaXiao,
Lounici2012}). In our paper \cite{Koltchinskii_Lounici_arxiv}, we proved that in the Gaussian 
case the size of the operator norm $\|\hat \Sigma-\Sigma\|_{\infty}$ can be completely 
characterized in terms of the effective rank ${\bf r}(\Sigma)$ of the true covariance $\Sigma$
and its operator norm $\|\Sigma\|_{\infty}$ and that the resulting non-asymptotic bounds 
are dimension-free (see theorems \ref{th_operator} and \ref{spectrum_sharper} below). 
This shows that $\hat \Sigma$ is an operator norm consistent estimator of $\Sigma$
provided that ${\bf r}(\Sigma)=o(n),$ which makes the effective rank $\mathbf{ r}(\Sigma)$
an important complexity parameter of the covariance estimation problem. 
This also provides a dimension-free framework for 
such problems and allows one to study them in a ``high-complexity" case (that is, when the effective
rank ${\bf r}(\Sigma)$ could be large) without imposing any structural assumptions on the true covariance 
such as, for instance, spiked covariance models \cite{Johnstone_Lu}. This approach has been 
developed in some detail in our recent papers \cite{Koltchinskii_Lounici_arxiv}, \cite{Koltchinskii_Lounici_bilinear}, \cite{ Koltchinskii_Lounici_HS}. 
The current paper continues this line of work by studying the asymptotic behavior
of several important statistics under the assumptions that both $n\to\infty$ and ${\bf r}(\Sigma)\to \infty,
{\bf r}(\Sigma)=o(n).$ This includes statistical estimators of bias of spectral projectors of $\hat \Sigma$
(empirical spectral projectors) as well as their squared Hilbert-Schmidt norm errors with a goal 
to develop ``studentized versions" of these statistics that could be (in principle) used for 
statistical inference.   
Before stating our main results, we provide in the next section a review of the results of 
papers \cite{Koltchinskii_Lounici_arxiv}, \cite{Koltchinskii_Lounici_bilinear}, \cite{ Koltchinskii_Lounici_HS}
that will be extensively used in what follows.

Throughout the paper, we write $A\lesssim B$ iff $A\leq CB$ for some absolute constant $C>0$
($A,B\geq 0$). $A\gtrsim B$ is equivalent to $B\lesssim A$ and $A\asymp B$ is equivalent 
to $A\lesssim B$ and $A\gtrsim B.$ Sometimes, the signs $\lesssim, \gtrsim$ and $\asymp$
could be provided with subscripts: for instance, $A\lesssim_{\gamma} B$ means that $A\leq CB$
with a constant $C$ that could depend on $\gamma.$ 

\section{Effective rank and concentration of empirical spectral projectors: a review of recent results}\label{sec:empirical}

The following recent result (see, \cite{Koltchinskii_Lounici_arxiv}) provides a complete characterization of the quantity 
${\mathbb E}\|\hat \Sigma-\Sigma\|_{\infty}$ in terms of the operator norm $\|\Sigma\|_{\infty}$
and the effective rank ${\bf r}(\Sigma)$ in the case of i.i.d. mean zero Gaussian observations. 

\begin{theorem}
\label{th_operator}
The following bound holds:
\begin{align}
\label{E1}
\E\|\hat\Sigma - \Sigma\|_{\infty} \asymp \|\Sigma\|_\infty \biggl[\sqrt{\frac{\mathbf{r}(\Sigma)}{n}}
\bigvee \frac{\mathbf{r}(\Sigma)}{n}\biggr].
\end{align}
\end{theorem}

In paper \cite{Koltchinskii_Lounici_arxiv}, it is also complemented by a concentration inequality for  
$\|\hat\Sigma - \Sigma\|_{\infty}$ around its expectation:

\begin{theorem}
\label{spectrum_sharper} 
There exists a constant 
$C>0$ such that for all $t\geq 1$ with probability at least $1-e^{-t},$ 
\begin{align}
 \label{con_con}
\Bigl|\|\hat\Sigma - \Sigma\|_{\infty}- {\mathbb E}\|\hat\Sigma - \Sigma\|_{\infty}\Bigr| \leq C\|\Sigma\|_\infty
\biggl[\biggl(\sqrt{\frac{\mathbf{r}(\Sigma)}{n}}\bigvee 1\biggr)\sqrt{\frac{t}{n}}\bigvee \frac{t}{n} 
\biggr].
\end{align}
\end{theorem}

It follows from (\ref{E1}) and (\ref{con_con}) that with some constant $C>0$ and with probability
at least $1-e^{-t}$ 
\begin{align}
\label{sha_sha}
\|\hat\Sigma - \Sigma\|_{\infty} \leq C\|\Sigma\|_\infty
\biggl[\sqrt{\frac{\mathbf{r}(\Sigma)}{n}} \bigvee \frac{\mathbf{r}(\Sigma)}{n} 
\bigvee \sqrt{\frac{t}{n}}\bigvee \frac{t}{n} 
\biggr],
\end{align}
which, in turn, implies that for all $p\geq 1$
\begin{align}
\label{E1_p}
\E^{1/p}\|\hat\Sigma - \Sigma\|_{\infty}^p \asymp_p
\|\Sigma\|_\infty \biggl[\sqrt{\frac{\mathbf{r}(\Sigma)}{n}}
\bigvee \frac{\mathbf{r}(\Sigma)}{n}\biggr].
\end{align}

These results showed that the sample covariance $\hat \Sigma$ is an operator norm consistent estimator of $\Sigma$ even in the cases when the effective rank ${\bf r}(\Sigma)$ becomes large as $n\to\infty,$
but ${\bf r}(\Sigma)=o(n)$ and $\|\Sigma\|_{\infty}$ remains bounded. Thus, it becomes of interest 
to study the behavior of spectral projectors of sample covariance $\hat \Sigma$ (that are of crucial importance in PCA) in such an asymptotic framework. This program has been partially implemented 
in papers \cite{Koltchinskii_Lounici_bilinear}, \cite{ Koltchinskii_Lounici_HS}. To state the main results 
of these papers (used in what follows), we will introduce some further definitions and notations. 

Let $\Sigma=\sum_{r\geq 1}\mu_r P_r$ be the spectral representation of covariance operator $\Sigma$
with distinct non zero eigenvalues $\mu_r, r\geq 1$ (arranged in decreasing order) and the corresponding 
spectral projectors $P_r, r\geq 1.$ Clearly, $P_r$ are finite rank projectors with ${\rm rank}(P_r)=:m_r$
being the multiplicity of the corresponding eigenvalue $\mu_r.$ Let $\sigma(\Sigma)$ be the spectrum 
of $\Sigma.$ Denote by $\bar g_r$ the distance from the eigenvalue $\mu_r$ to the rest of the spectrum 
$\sigma(\Sigma)\setminus \{\mu_r\}$ (the $r$-th ``spectral gap"). 
It will be also convenient to consider 
the non zero eigenvalues $\sigma_j(\Sigma), j\geq 1$ of $\Sigma$ arranged in nondecreasing order
and \it repeated with their multiplicities \rm  (in the case when the number of non zero eigenvalues 
is finite, we extend this sequence by zeroes).  With this notation, let $\Delta_r:=\{j: \sigma_j(\Sigma)=\mu_r\}, r\geq 1$ and denote by $\hat P_r$ the orthogonal projector onto the linear span of eigenspaces of $\hat \Sigma$ corresponding to its eigenvalues $\{\sigma_j(\hat \Sigma): j\in \Delta_r\}.$ It easily follows from a well known 
inequality due to Weyl that
$$
\sup_{j\geq 1}|\sigma_j(\hat \Sigma)-\sigma_j(\Sigma)|\leq \|\hat \Sigma-\Sigma\|_{\infty}.
$$
If $\|\hat \Sigma-\Sigma\|_{\infty}<\frac{\bar g_r}{2},$ this immediately implies that the eigenvalues 
$\{\sigma_j(\hat \Sigma):j\in \Delta_r\}$ form a ``cluster" that belongs to the interval 
$(\mu_r-\frac{\bar g_r}{2},\mu_r+\frac{\bar g_r}{2})$ and that is separated from the rest of the spectrum of $\hat \Sigma$ in the sense
that $\sigma_j(\hat \Sigma)\not\in (\mu_r-\frac{\bar g_r}{2},\mu_r+\frac{\bar g_r}{2})$ for all $j\not\in \Delta_r.$ 
In this case, $\hat P_r$ becomes a natural estimator of $P_r.$ It could be viewed as a random 
perturbation of $P_r$ and the following result, closely related to basic facts of perturbation theory
(see \cite{Kato1}), could be found in \cite{Koltchinskii_Lounici_bilinear} (see Lemmas 1 and 2 there). 

\begin{lemma}\label{lem-pert-spectral}
Let $E:=\hat \Sigma-\Sigma.$
%Suppose condition (\ref{bound_on_E}) holds.
The following bound holds: 
\begin{align}
\label{bd_1}
\|\hat P_r-P_r\|_\infty \leq 4\frac{\|E\|_{\infty}}{\bar g_r}.
\end{align}
Moreover, denote 
$$
C_r:= \sum_{s\neq r}\frac{1}{\mu_r-\mu_s}P_s.
$$
Then
\begin{equation}
\label{perture}
\hat P_r-P_r =L_r(E)+ S_r(E),
\end{equation}
where 
\begin{align}
\label{linear_perturb}
L_r(E):=C_r E P_r + P_r E C_r
\end{align}
and 
\begin{equation}
\label{remainder_A}
\|S_r(E)\|_{\infty}\leq 14 \biggl(\frac{\|E\|_{\infty}}{\bar g_r}\biggr)^2.
\end{equation}
\end{lemma}

\begin{remark}
In the case when $0$ is an eigenvalue of $\Sigma,$ it is convenient to extend the sum 
in the definition of operator $C_r$ to $s=\infty$ with $\mu_{\infty}=0$ (see, for instance,
the proof of Lemma \ref{lem:pert-representations}). Note, however, that $P_{\infty}\Sigma = \Sigma P_{\infty}=0$ and $P_{\infty}\hat \Sigma=\hat \Sigma P_{\infty}=0.$ Thus, this additional 
term in the definition of $C_r$ does not have any impact on $L_r(E)$ (and on the parameters
$A_r(\Sigma), B_r(\Sigma)$ introduced below). 
\end{remark}

This result essentially shows that the difference $\hat P_r-P_r$ can be represented as a sum 
of two terms, a linear term with respect to $E=\hat \Sigma-\Sigma$ denoted by $L_r(E)$ and 
the remainder term $S_r(E)$ for which bound (\ref{remainder_A}) (quadratic with respect 
to $\|E\|_{\infty}^2$) holds. The linear term $L_r(E)$ could be further represented as a sum 
of i.i.d. mean zero random operators:
$$
L_r(E)=n^{-1}\sum_{j=1}^n (P_r X_j\otimes C_r X_j+ C_r X_j\otimes P_r X_j),
$$
which easily implies simple concentration bounds and asymptotic normality results 
for this term. On the other hand, it follows from theorems \ref{th_operator} and \ref{spectrum_sharper}
that with probability at least $1-e^{-t}$
$$
\|S_r(E)\|_{\infty}\lesssim \frac{\|\Sigma\|_{\infty}^2}{\bar g_r^2}\biggl[\frac{\mathbf{r}(\Sigma)}{n} \bigvee \biggl(\frac{\mathbf{r}(\Sigma)}{n}\biggr)^2 
\bigvee \frac{t}{n}\bigvee \biggl(\frac{t}{n}\biggr)^2 
\biggr],
$$
implying that $\|S_r(E)\|_{\infty}=o_{\mathbb P}(1)$ under the assumption ${\bf r}(\Sigma)=o(n)$
and $\|S_r(E)\|_{\infty}=o_{\mathbb P}(n^{-1/2})$ under the assumption ${\bf r}(\Sigma)=o(n^{1/2})$
(in both cases, provided that $\frac{\|\Sigma\|_{\infty}}{\bar g_r}$ remains bounded). Bound on the remainder 
term $S_r(E)$ of the order $o_{\mathbb P}(n^{-1/2})$ makes this term negligible if the linear term $L_r(E)$ 
converges to zero with the rate $O_{\mathbb P}(n^{-1/2})$ (the standard rate of the central limit theorem).
A more subtle analysis of bilinear forms $\langle S_r(E)u,v\rangle, u,v\in {\mathbb H}$ given in \cite{Koltchinskii_Lounici_bilinear} showed that the bilinear forms concentrate around their expectations
at a rate $o_{\mathbb P}(n^{-1/2})$ provided that ${\bf r}(\Sigma)=o(n)$ (which is much weaker than the 
assumption ${\bf r}(\Sigma)=o(n^{1/2})$ needed for the operator norm $\|S_r(E)\|_{\infty}$ to be of 
the order $o_{\mathbb P}(n^{-1/2})$). More precisely, the following result was proved for the 
operator
$$
R_r:= R_r(E):=S_r(E)-{\mathbb E}S_r(E)=\hat P_r - P_r - {\mathbb E}(\hat P_r-P_r)-L_r(E)
$$
(see Theorem 3 in \cite{Koltchinskii_Lounici_bilinear}):

\begin{theorem}
\label{technical_1}
Suppose that, for some $\gamma\in (0,1),$
\begin{equation}
\label{deltanleq}
{\mathbb E}\|\hat \Sigma-\Sigma\|_{\infty}\leq (1-\gamma)\frac{\bar g_r}{2}.
\end{equation}
Then, there exists a constant $D_{\gamma}>0$ such that, for all $u,v\in {\mathbb H}$
and for all $t\geq 1,$ 
the following bound holds with probability at least $1-e^{-t}:$ 
\begin{equation}
\label{remaind}
|\langle R_r u,v\rangle |\leq 
D_\gamma \frac{\|\Sigma\|_{\infty}^2}{\bar g_r^2}
\biggl(\sqrt{\frac{{\bf r}(\Sigma)}{n}}\bigvee \sqrt{\frac{t}{n}}\bigvee \frac{t}{n}\biggr)
\sqrt{\frac{t}{n}}\|u\|\|v\|.
\end{equation}
\end{theorem}

Condition (\ref{deltanleq}) (along with concentration bound of Theorem \ref{spectrum_sharper}) essentially guarantees that $\|\hat \Sigma-\Sigma\|_{\infty}<\frac{\bar g_r}{2}$ with a high probability, which makes 
the empirical spectral projector $\hat P_r$ a small random perturbation of the true spectral projector $P_r$
and allows us to use the tools of perturbation theory.  Theorem \ref{technical_1} easily implies the 
following concentration bound for bilinear forms $\langle \hat P_r u,v\rangle:$

\begin{corollary}
\label{concent_bilin}
Under the assumption of Theorem \ref{technical_1}, with some constants $D, D_{\gamma}>0,$ for all $u,v\in {\mathbb H}$ and for all $t\geq 1$ with probability at least $1-e^{-t},$
\begin{align}
\label{bdPr}
&
\nonumber
\Bigl|\Bigl\langle \hat P_r-{\mathbb E}\hat P_r u,v\Bigr\rangle\Bigr|
\leq D \frac{\|\Sigma\|_{\infty}}{\bar g_r}\sqrt{\frac{t}{n}}\|u\|\|v\|
\\
&
+
D_\gamma \frac{\|\Sigma\|_{\infty}^2}{\bar g_r^2}\biggl(\sqrt{\frac{{\bf r}(\Sigma)}{n}}\bigvee \sqrt{\frac{t}{n}}\bigvee \frac{t}{n}\biggr)
\sqrt{\frac{t}{n}}\|u\|\|v\|.
\end{align}
\end{corollary}

Moreover, it is easy to see that if both $u$ and $v$ are either in the eigenspace of $\Sigma$ corresponding 
to the eigenvalue $\mu_r,$ or in the orthogonal complement of this eigenspace, then the first term in 
the right hand side of bound (\ref{bdPr}) could be dropped and the bound reduces to its second term. 

In addition to this, in \cite{Koltchinskii_Lounici_bilinear}, the asymptotic normality of bilinear 
forms $\langle \hat P_r-{\mathbb E}\hat P_r u,v\rangle, u,v\in {\mathbb H}$
was also proved in an asymptotic framework where $n\to \infty$ and ${\bf r}(\Sigma)=o(n).$ 

Another important question studied in \cite{Koltchinskii_Lounici_bilinear} concerns the structure of 
the bias ${\mathbb E}\hat P_r-P_r$ of empirical spectral projector $\hat P_r.$ Namely, it was proved 
that the bias can be represented as the sum of two terms, the main term $P_r({\mathbb E}\hat P_r-P_r)P_r$
being ``aligned" with the projector $P_r$ and the remainder $T_r$ being of a smaller order in the 
operator norm (provided that ${\bf r}(\Sigma)=o(n)$). More specifically, the following result was proved (see Theorem 4 in \cite{Koltchinskii_Lounici_bilinear}). 

\begin{theorem}
\label{bd_bias}
Suppose that, for some $\gamma\in (0,1),$ condition (\ref{deltanleq}) holds.
Then, there exists a constant $D_{\gamma}>0$ such that
$$
{\mathbb E}\hat P_r-P_r= P_r({\mathbb E}\hat P_r-P_r)P_r+T_r
$$
with $P_rT_rP_r=0$ and 
\begin{equation}
\label{T_r}
\|T_r\|_{\infty}\leq D_{\gamma}
\frac{m_r\|\Sigma\|_{\infty}^2}{\bar g_r^2}
\sqrt{\frac{{\bf r}(\Sigma)}{n}}\frac{1}{\sqrt{n}}.
\end{equation}
\end{theorem}

In the case when $m_r={\rm rank}(P_r)=1$ (so, $\mu_r$ is an eigenvalue of $\Sigma$ of multiplicity $1$),
the structure of the bias becomes especially simple. Let $P_r=\theta_r\otimes \theta_r,$ where $\theta_r$
is a unit norm eigenvector of $\Sigma$ corresponding to $\mu_r.$ Then it is easy to see that 
\begin{equation}
\label{Ehat}
{\mathbb E}\hat P_r -P_r= b_rP_r+T_r
\end{equation}
with $b_r= \langle ({\mathbb E}\hat P_r-P_r)\theta_r,\theta_r\rangle$
and $T_r$ defined in Theorem \ref{bd_bias}. Moreover,
$$
b_r=\langle {\mathbb E}\hat P_r-P_r, \theta_r\otimes \theta_r\rangle=
{\mathbb E}\langle \hat \theta_r,\theta_r\rangle^2-1,
$$
implying that $b_r\in [-1,0].$ Thus, parameter $b_r$ is an important characteristic 
of the bias of empirical spectral projector $\hat P_r.$ It was shown in \cite{Koltchinskii_Lounici_bilinear}
that, under the assumption ${\bf r}(\Sigma)\lesssim n,$
\begin{equation}
\label{b_r_bd}
|b_r|\lesssim \frac{\|\Sigma\|_{\infty}^2}{\bar g_r^2}\frac{{\bf r}(\Sigma)}{n}.
\end{equation}
Note that this upper bound is larger than upper bound (\ref{T_r}) on the remainder 
$\|T_r\|_{\infty}$ by a factor $\sqrt{{\bf r}(\Sigma)}.$ 

Let now $\hat P_r=\hat \theta_r \otimes \hat \theta_r$ with a unit norm eigenvector $\hat \theta_r$
of $\hat \Sigma.$ Since the vectors $\hat \theta_r, \theta_r$ are defined only up to their signs, 
assume without loss of generality that $\langle \hat \theta_r,\theta_r\rangle\geq 0.$ 
The following result, proved in \cite{Koltchinskii_Lounici_bilinear} (see Theorem 6), 
shows that the linear forms $\langle \hat \theta_r, u\rangle$ have ``Bernstein type"
concentration around $\sqrt{1+b_r}\langle \theta_r, u\rangle$ with deviations of the order 
$O_{\mathbb P}(n^{-1/2}).$

\begin{theorem}
\label{th:hattheta}
Suppose that condition (\ref{deltanleq}) holds for some $\gamma\in (0,1)$ and also that 
\begin{equation}
\label{cond_b_r}
1+b_r\geq \frac{\gamma}{2}.
\end{equation}
Then, there exists a constant $C_{\gamma}>0$ such that for all $t\geq 1$ with probability at least 
$1-e^{-t}$
$$
\Bigl|\Bigl\langle \hat \theta_r-\sqrt{1+b_r}\theta_r,u\Bigr\rangle\Bigr|
\leq C_{\gamma} \frac{\|\Sigma\|_{\infty}}{\bar g_r}
\biggl(\sqrt{\frac{t}{n}}\bigvee \frac{t}{n}\biggr)\|u\|.
$$  
\end{theorem}

Thus, if one constructs a proper estimator of the bias parameter $b_r,$ it would 
be possible to improve a ``naive estimator" $\langle \hat \theta_r, u\rangle$
of linear form $\langle \theta_r,u\rangle$ by reducing its bias. A version of such 
estimator based on the double sample $X_1,\dots, X_n, \tilde X_1,\dots, \tilde X_n$ of 
i.i.d. copies of $X$ was suggested in \cite{Koltchinskii_Lounici_bilinear}. If 
$\tilde \Sigma=\tilde \Sigma_n$ denotes the sample covariance based on 
$\tilde X_1,\dots, \tilde X_n$ (the second subsample) and $\tilde P_r=\tilde \theta_r\otimes \tilde \theta_r$
denotes the corresponding empirical spectral projector (estimator of $P_r$), then 
the estimator $\hat b_r$ of the bias parameter $b_r$ is defined as follows:
$$
\hat b_r := \langle \hat \theta_r, \tilde \theta_r\rangle -1,
$$ 
where the signs of $\hat \theta_r, \tilde \theta_r$ are chosen so that $\langle \hat \theta_r, \tilde \theta_r\rangle\geq 0.$ Based on estimator $\hat b_r,$ one can also define a bias corrected estimator 
$\check \theta_r:= \frac{\hat \theta_r}{\sqrt{1+\hat b_r}}$ (which is not necessarily a unit vector)
and prove the following result, showing that $\langle \check \theta_r,u\rangle$ is a $\sqrt{n}$-consistent 
estimator of $\langle \theta_r, u\rangle$ (at least in the case when ${\bf r}(\Sigma)\leq cn$ for a sufficiently 
small constant $c$):

\begin{proposition}\label{prop:tildetheta}
Under the assumptions and notations of Theorem \ref{th:hattheta}, for some 
constant $C_{\gamma}>0$ with probability at least $1-e^{-t},$
\begin{equation}
\label{estim_br}
|\hat b_r-b_r|
\leq 
C_\gamma \frac{\|\Sigma\|_{\infty}^2}{\bar g_r^2}
\biggl(\sqrt{\frac{{\bf r}(\Sigma)}{n}}\bigvee \sqrt{\frac{t}{n}}\bigvee \frac{t}{n}
\biggr)
\biggl(\sqrt{\frac{t}{n}}\bigvee \frac{t}{n}\biggr),
\end{equation}
and, for all $u\in {\mathbb H},$ with the same probability
\begin{equation}
\label{tildetheta}
\Bigl|\langle \check \theta_r-\theta_r,u\rangle\Bigr| 
\leq C_{\gamma} \frac{\|\Sigma\|_{\infty}}{\bar g_r}
\biggl(\sqrt{\frac{t}{n}}\bigvee \frac{t}{n}\biggr)\|u\|.
\end{equation}
\end{proposition}

In addition to this, asymptotic normality of $\langle \check \theta_r,u\rangle$ was also proved in 
\cite{Koltchinskii_Lounici_bilinear} under the assumption that ${\bf r}(\Sigma)=o(n).$

Finally, we will discuss the results on normal approximation of the (squared) Hilbert--Schmidt norms
$\|\hat P_r-P_r\|_2^2$ for an empirical spectral projector $\hat P_r$ obtained in \cite{Koltchinskii_Lounici_HS}. It was shown in this paper that, in the case when ${\bf r}(\Sigma)=o(n),$  the size of the expectation ${\mathbb E}\|\hat P_r-P_r\|_2^2$ could be characterized by 
the quantity $A_r(\Sigma):=2{\rm tr}(P_r\Sigma P_r){\rm tr}(C_r\Sigma C_r)$ 
(which, under mild assumption, is of the same order as ${\bf r}(\Sigma)$):
$$
{\mathbb E}\|\hat P_r-P_r\|_2^2= (1+o(1))\frac{A_r(\Sigma)}{n}.
$$ 
A similar parameter characterizing the size of the variance ${\rm Var}(\|\hat P_r-P_r\|_2^2)$
is defined as $B_r(\Sigma):=2\sqrt{2}\|P_r\Sigma P_r\|_2\|C_r\Sigma C_r\|_2.$ Namely, 
the following result holds (Theorem 7 in \cite{Koltchinskii_Lounici_HS}):

\begin{theorem}
\label{th:varvar}
Suppose condition (\ref{deltanleq}) holds for some $\gamma\in (0,1).$ 
Then the following bound holds with some constant $C_{\gamma}>0:$
\begin{align}
\label{variance_bd}
&
%\nonumber
\left|\frac{n}{B_r(\Sigma)}{\rm Var}^{1/2}(\|\hat P_r-P_r\|_2^2)-1\right|
\leq 
%\\
%&
C_{\gamma} m_r\frac{\|\Sigma\|_{\infty}^3}{\bar g_r^3}\frac{{\bf r}(\Sigma)}{B_r(\Sigma)\sqrt{n}}+
\frac{m_r+1}{n}.
\end{align}
\end{theorem}

If $\frac{\|\Sigma\|_{\infty}}{\bar g_r}$ and $m_r$ are bounded and 
$\frac{{\bf r}(\Sigma)}{B_r(\Sigma)\sqrt{n}}\to 0,$ this implies that 
$$
{\rm Var}(\|\hat P_r-P_r\|_2^2)=(1+o(1))\frac{B_r^2(\Sigma)}{n^2}.
$$

The main result of \cite{Koltchinskii_Lounici_HS} is the following normal approximation bounds  
for $\|\hat P_r-P_r\|_2^2:$

\begin{theorem}
\label{th:normal_approx}
Suppose that, for some constants $c_1,c_2>0,$ $m_r\leq c_1$ and $\|\Sigma\|_{\infty}\leq c_2 \bar g_r.$ 
Suppose also condition (\ref{deltanleq}) holds with some $\gamma\in (0,1).$
Then, the following bounds hold with some constant $C>0$ depending only on $\gamma, c_1, c_2:$
\begin{align}
\label{normal_approx_A}
&
\nonumber
\sup_{x\in {\mathbb R}}
\left|{\mathbb P}\left\{\frac{n}{B_r(\Sigma)}\left(\|\hat P_r-P_r\|_2^2-{\mathbb E}\|\hat P_r-P_r\|_2^2\right)\leq x\right\}-\Phi(x)\right|
\\
&
\leq 
C\left[\frac{1}{B_r(\Sigma)}+\frac{{\bf r}(\Sigma)}{B_r(\Sigma)\sqrt{n}}\sqrt{\log \left(\frac{B_r(\Sigma)\sqrt{n}}{{\bf r}(\Sigma)}\bigvee 2\right) }\right]
\end{align}
and
\begin{align}
\label{normal_approx}
&
\nonumber
\sup_{x\in {\mathbb R}}
\left|{\mathbb P}\left\{\frac{\|\hat P_r-P_r\|_2^2-{\mathbb E}\|\hat P_r-P_r\|_2^2}{{\rm Var}^{1/2}(\|\hat P_r-P_r\|_2^2)}\leq x\right\}-\Phi(x)\right|
\\
&
\leq 
C\left[\frac{1}{B_r(\Sigma)}+\frac{{\bf r}(\Sigma)}{B_r(\Sigma)\sqrt{n}}\sqrt{\log \left(\frac{B_r(\Sigma)\sqrt{n}}{{\bf r}(\Sigma)}\bigvee 2\right) }\right],
\end{align}
where $\Phi(x)$ denotes the distribution function of standard normal random variable.
\end{theorem}

These bounds show that asymptotic normality of properly normalized statistic $\|\hat P_r-P_r\|_2^2$
holds provided that $n\to\infty,$ $B_r(\Sigma)\to \infty$ and 
$\frac{{\bf r}(\Sigma)}{B_r(\Sigma)\sqrt{n}}\to 0.$ In the case 
of $p$-dimensional spiked covariance models (with a fixed number of spikes), these 
conditions boil down to $n\to \infty,$ $p\to \infty$ and $p=o(n).$ 

\section{Main results}

We start this section with introducing a precise asymptotic framework in which ${\bf r}(\Sigma)\to \infty$
as $n\to\infty.$ It is assumed that an observation $X=X^{(n)}$ is sampled from 
from a Gaussian distributions in ${\mathbb H}$ with mean zero and covariance $\Sigma=\Sigma^{(n)}.$ 
The data consists on $n$ i.i.d. copies of $X^{(n)}:$
$X_1=X_1^{(n)},\dots, X_n=X_n^{(n)}$ and the sample covariance $\hat \Sigma_n$ is based on $(X_1^{(n)},\dots, X_n^{(n)}).$ As before, $\mu_{r}^{(n)}, r\geq 1$ denote distinct nonzero eigenvalues of $\Sigma^{(n)}$ arranged in decreasing order and $P_{r}^{(n)}, r\geq 1$ the corresponding spectral projectors.
Let $\Delta_{r}^{(n)}:=\{j: \sigma_j(\Sigma^{(n)})=\mu_{r}^{(n)}\}$
and let $\hat P_{r}^{(n)}$ be the orthogonal projector on the linear span 
of eigenspaces corresponding to the eigenvalues $\{\sigma_j(\hat \Sigma_n), j\in \Delta_{r}^{(n)}\}.$

The goal is to estimate the spectral projector $P^{(n)}=P_{r_n}^{(n)}$ corresponding 
to the eigenvalue $\mu^{(n)}=\mu_{r_n}^{(n)}$ of $\Sigma^{(n)}$
with multiplicity $m^{(n)}=m_{r_n}^{(n)}$ and with spectral gap  
$\bar g^{(n)}=\bar g_{r_n}^{(n)}.$ 
Define $C^{(n)}=C_{r_n}^{(n)}:= \sum_{s\neq r_n}\frac{1}{\mu_{r_n}^{(n)}-\mu_s^{(n)}}P_s^{(n)}$ and 
let
$$
B_n:=B_{r_n}(\Sigma^{(n)}):= 2\sqrt{2}\|C^{(n)}\Sigma^{(n)}C^{(n)}\|_2\|P^{(n)}\Sigma^{(n)}P^{(n)}\|_2.
$$
%is asymptotically standard normal. 

\begin{assumption}
\label{ass_sigma_n}
Suppose the following conditions hold:
\begin{equation}
\label{ass_sigma_n_1}
\sup_{n\geq 1}m^{(n)}<+\infty;
\end{equation}
\begin{equation}
 \label{ass_sigma_n_2}
 \sup_{n\geq 1}\frac{\|\Sigma^{(n)}\|_{\infty}}{\bar g^{(n)}}<+\infty; 
\end{equation}
%\begin{equation}
%\label{ass_sigma_n_2}
%\sup_{n\geq 1}m^{(n)}<+\infty; 
%\end{equation}
%\begin{equation}
%\label{ass_sigma_n_3}
%B_n\to \infty\ {\rm as}\ n\to\infty; 
%\end{equation}
\begin{equation}
\label{ass_sigma_n_3}
B_n\to \infty\ {\rm as}\ n\goin;
\end{equation}
\begin{equation}
\label{ass_sigma_n_3}
\frac{{\bf r}(\Sigma^{(n)})}{B_n\sqrt{n}}\to 0\ {\rm as}\ n\goin.
\end{equation}
\end{assumption}

Assumption \ref{ass_sigma_n}
easily implies that 
$$
{\bf r}(\Sigma^{(n)})\to \infty,\ \ {\bf r}(\Sigma^{(n)})=o(n)\ {\rm as}\ n\to\infty.
$$
Also, under mild additional conditions, $B_n\asymp \|\Sigma^{(n)}\|_2.$

The following fact is an immediate consequence of bound (\ref{variance_bd}) and 
Theorem \ref{th:normal_approx}. 

\begin{proposition}
\label{th_2_norm}
Under Assumption \ref{ass_sigma_n},
$$
\mathrm{Var}(\|\hat P^{(n)}-P^{(n)}\|_2^2) =\left( \frac{B_n}{n}\right)^2 (1+ o(1)).
$$
In addition, 
\begin{align}\label{CLT-sample proj}
\frac{n\Bigl(\|\hat P^{(n)}-P^{(n)}\|_2^2-{\mathbb E}\|\hat P^{(n)}-P^{(n)}\|_2^2\Bigr)}{B_n}\overset{d}\longrightarrow Z
\end{align}
and 
\begin{align}
\label{CLT-sample proj''}
\frac{\Bigl(\|\hat P^{(n)}-P^{(n)}\|_2^2-{\mathbb E}\|\hat P^{(n)}-P^{(n)}\|_2^2\Bigr)}
{\sqrt{\mathrm{Var}(\|\hat P^{(n)}-P^{(n)}\|_2^2)}} \overset{d}\longrightarrow Z \ {\rm as}\ n\to\infty,
\end{align}
$Z$ being a standard normal random variable.
\end{proposition}

Our main goal is to develop a version of these asymptotic results for squared Hilbert--Schmidt 
norm error $\|\hat P^{(n)}-P^{(n)}\|_2^2$ of empirical spectral projector $P^{(n)}$ with a {\it data driven normalization} that, in principle, could lead to constructing confidence sets and statistical tests
for spectral projectors of covariance operator under Assumption \ref{ass_sigma_n}. 
This will be done only in the case when the target spectral projector $P^{(n)}$ is of rank $1$
(that is, $\mu^{(n)}$ is the eigenvalue of multiplicity $m^{(n)}=1$). This problem is also related 
to estimation of the bias parameter $b^{(n)}=b_{r_n}^{(n)}$ of empirical spectral projector $\hat P^{(n)}.$
This parameter and its estimator $\hat b^{(n)}=\hat b_{r_n}^{(n)}$ were introduced in Section \ref{sec:empirical}. In particular, we will prove the asymptotic normality of estimator $\hat b^{(n)}$
with a proper normalization that depends on unknown covariances $\Sigma^{(n)}$ and derive 
the limit distribution of $\hat b^{(n)}$ with a data-driven normalization. 

Let $P^{(n)}=\theta^{(n)}\otimes \theta^{(n)}$ and $\hat P^{(n)}=\hat \theta^{(n)}\otimes \hat \theta^{(n)}$
for unit vectors $\theta^{(n)}, \hat \theta^{(n)}.$ To define the estimator $\hat b^{(n)},$ we need an additional 
independent sample $\tilde X_1^{(n)}, \dots, \tilde X_n^{(n)}$ consisting of i.i.d. copies of $X^{(n)}.$ Let $\tilde \Sigma_n$ denote the sample covariance based on $(\tilde X_1^{(n)},\dots,  \tilde X_n^{(n)})$ and 
let $\tilde P^{(n)}=\tilde \theta^{(n)}\otimes \tilde \theta^{(n)}$ be its empirical spectral projector corresponding to $P^{(n)}.$ It will be assumed that the signs of $\hat \theta^{(n)}, \tilde \theta^{(n)}$ are chosen 
in such a way that $\langle \hat \theta^{(n)}, \tilde \theta^{(n)}\rangle\geq 0.$ Define
$$
\hat b^{(n)}= \langle \hat \theta^{(n)}, \tilde \theta^{(n)}\rangle-1.
$$

\begin{theorem}
\label{th_odin}
Under Assumption \ref{ass_sigma_n}, 
$$
\frac{2n}{B_n}(\hat b^{(n)}-b^{(n)}) \overset{d}\longrightarrow Z \ {\rm as}\ n\to\infty,
$$
$Z$ being a standard normal random variable.
\end{theorem}

In order to use this asymptotic normality result for statistical inference about bias parameter 
$b^{(n)},$ one has to find a way to estimate the normalizing factor $\frac{2n}{B_n}$ that depends 
on unknown covariance $\Sigma^{(n)}.$ By the first claim of Proposition \ref{th_2_norm},
under Assumption \ref{ass_sigma_n},
$$
\frac{2n}{B_n} \sim \frac{2}{{\rm Var}^{1/2}(\|\hat P^{(n)}-P^{(n)}\|_2^2)}\ {\rm as}\ n\to\infty.
$$
Thus, equivalently, we need to estimate the variance ${\rm Var}(\|\hat P^{(n)}-P^{(n)}\|_2^2).$
Note that 
$$
{\rm Var}(\|\hat P^{(n)}-P^{(n)}\|_2^2)= 
{\rm Var}\Bigl(\|\hat P^{(n)}\|_2^2+\|P^{(n)}\|_2^2-2\langle \hat P^{(n)}, P^{(n)}\rangle\Bigr)
$$
$$
=
{\rm Var}\Bigl(2-2\langle \hat P^{(n)}, P^{(n)}\rangle\Bigr)= 
4{\rm Var}\Bigl(\langle \hat P^{(n)}, P^{(n)}\rangle\Bigr)
=2 {\mathbb E}\Bigl(\langle \hat P^{(n)}, P^{(n)}\rangle-\langle \tilde P^{(n)}, P^{(n)}\rangle\Bigr)^2.
$$
To estimate the right hand side, consider the third independent sample $\bar X_1^{(n)},\dots, \bar X_n^{(n)}$
consisting of $n$ independent copies of $X^{(n)}$ and denote by $\bar \Sigma_n$ the sample covariance 
based on $(\bar X_1^{(n)},\dots, \bar X_n^{(n)})$ and by $\bar P^{(n)}=\bar \theta^{(n)}\otimes \bar \theta^{(n)}$ its empirical spectral projector corresponding to $P^{(n)}.$ Assume that the sign of $\bar \theta^{(n)}$ is chosen in such a way that $\langle \tilde \theta^{(n)},\bar \theta^{(n)}\rangle \geq 0$
and define 
$$
\tilde b^{(n)}:= \langle \tilde \theta^{(n)},\bar \theta^{(n)}\rangle -1.
$$
We will use 
$$
\langle \hat P^{(n)}, \tilde P^{(n)}\rangle= \langle \hat \theta^{(n)},\tilde \theta^{(n)}\rangle^2=
(1+\hat b^{(n)})^2
$$
as an ``estimator" of $\langle \hat P^{(n)}, P^{(n)}\rangle$ and 
$$
\langle \tilde P^{(n)}, \bar P^{(n)}\rangle= \langle \tilde \theta^{(n)},\bar \theta^{(n)}\rangle^2=
(1+\tilde b^{(n)})^2
$$ 
as an ``estimator" of $\langle \tilde P^{(n)}, P^{(n)}\rangle.$
To estimate ${\rm Var}(\|\hat P^{(n)}-P^{(n)}\|_2^2)\sim \frac{B_n^2}{n^2},$ one can try to use the statistic 
$2\Bigl((1+\hat b^{(n)})^2-(1+\tilde b^{(n)})^2\Bigr)^2.$ In fact, it turns out that 
the sequence 
$$
\frac{n}{B_n}\Bigl((1+\hat b^{(n)})^2-(1+\tilde b^{(n)})^2\Bigr)
\sim 
\frac{(1+\hat b^{(n)})^2-(1+\tilde b^{(n)})^2}{{\rm Var}^{1/2}(\|\hat P^{(n)}-P^{(n)}\|_2^2)}
$$ 
is asymptotically normal with mean zero and variance $\frac{3}{2}$
and 
$$
\frac{{\mathbb E}\Bigl|(1+\hat b^{(n)})^2-(1+\tilde b^{(n)})^2\Bigr|}{{\rm Var}^{1/2}(\|\hat P^{(n)}-P^{(n)}\|_2^2)}
\to \sqrt{\frac{3}{2}}{\mathbb E}|Z|= \sqrt{\frac{3}{\pi}}\ {\rm as}\ n\to\infty.
$$
Therefore, it might be more natural to view $\frac{\pi}{3}\Bigl((1+\hat b^{(n)})^2-(1+\tilde b^{(n)})^2\Bigr)^2$
as an estimator of the variance ${\rm Var}(\|\hat P^{(n)}-P^{(n)}\|_2^2).$ In any case, we are more interested in a data driven version 
of Theorem \ref{th_odin} given below. 

Given $\alpha\in {\mathbb R}, \beta>0,$ let $Y_{\alpha,\beta}$ denote a random variable with density 
$$
\frac{1}{2}\Bigl[\frac{1}{\beta}f\Bigl(\frac{x-\alpha}{\beta}\Bigr)+\frac{1}{\beta}f\Bigl(\frac{x+\alpha}{\beta}\Bigr)\Bigr],
$$ 
$f(x):=\frac{1}{\pi(1+x^2)}, x\in {\mathbb R}$ being the standard Cauchy 
density. The distribution of $Y_{\alpha,\beta}$ is a mixture of two rescaled 
Cauchy densities with locations $\pm \alpha$ and with equal mixing probabilities. 
This distribution (with proper choices of parameters $\alpha, \beta$) occurs naturally as the distribution of 
the ration $\frac{\xi}{|\eta|}$ for mean zero normal random variables $\xi, \eta.$ Namely, the following
(probably, well known) fact holds. Its proof is rather elementary and is left to the reader. 

\begin{proposition}
\label{Cauchy}
Suppose $\xi, \eta$ are mean zero normal random variables with ${\mathbb E}\xi^2=\sigma_{\xi}^2>0,$
${\mathbb E}\eta^2=\sigma_{\eta}^2>0$ and with correlation coefficient $\rho.$ Then 
$$\frac{\xi}{|\eta|}\stackrel{d}{=} Y_{\alpha,\beta}$$
with $\alpha := \frac{\sigma_{\xi}}{\sigma_{\eta}}\rho$ 
and $\beta:= \frac{\sigma_{\xi}}{\sigma_{\eta}}\sqrt{1-\rho^2}.$ 
\end{proposition}

We now state a data-driven version of Theorem \ref{th_odin}.

\begin{theorem}
\label{th_dva}
Under Assumption \ref{ass_sigma_n}, 
$$
\frac{2(\hat b^{(n)}-b^{(n)})}{\Bigl|(1+\hat b^{(n)})^2-(1+\tilde b^{(n)})^2\Bigr|} \overset{d}\longrightarrow Y_{\alpha,\beta} \ {\rm as}\ n\to\infty,
$$
where $\alpha:=\frac{1}{2},$ $\beta:=\sqrt{\frac{5}{12}}.$
\end{theorem} 

Quite similarly, we will determine the asymptotic distribution of statistic $\|\hat P^{(n)}-P^{(n)}\|_2^2$
with a data-driven normalization. First note that 
\begin{align}
\label{risk-bias}
&
\nonumber
{\mathbb E}\|\hat P^{(n)}-P^{(n)}\|_2^2=
{\mathbb E}\Bigl(\|\hat P^{(n)}\|_2^2+\|P^{(n)}\|_2^2-2\langle \hat P^{(n)}, P^{(n)}\rangle\Bigr)
={\mathbb E}\Bigl(2-2\langle \hat P^{(n)}, P^{(n)}\rangle\Bigr)
\\
&
=
2-2\langle {\mathbb E}\hat P^{(n)}, P^{(n)}\rangle
=2-2(1+b^{(n)})\langle P^{(n)}, P^{(n)}\rangle =
-2b^{(n)}
\end{align}
(see Theorem \ref{bd_bias} and the comments after this theorem). 
In the data-driven version of (\ref{CLT-sample proj''})
we will replace ${\mathbb E}\|\hat P^{(n)}-P^{(n)}\|_2^2$ by 
its estimator $-2\hat b^{(n)}$ and the standard deviation ${\rm Var}^{1/2}(\|\hat P^{(n)}-P^{(n)}\|_2^2)$ 
by $\Bigl|(1+\hat b^{(n)})^2-(1+\tilde b^{(n)})^2\Bigr|.$
This yields the following result. 

\begin{theorem}
\label{th_tri}
Under Assumption \ref{ass_sigma_n}, 
$$
\frac{\|\hat P^{(n)}-P^{(n)}\|_2^2+2\hat b^{(n)}}{\Bigl|(1+\hat b^{(n)})^2-(1+\tilde b^{(n)})^2\Bigr|} \overset{d}\longrightarrow Y_{\alpha,\beta} \ {\rm as}\ n\to\infty,
$$
where $\alpha:=\frac{5}{6},$ $\beta:=\frac{\sqrt{47}}{6}.$
\end{theorem} 

\section{Proofs: preliminary lemmas}

We start with preliminary results that will be formulated in the  ``non-asymptotic framework" of 
Section \ref{sec:empirical} and the notations of that section will be used.
Recall that $X_1,\dots ,X_n$ and $\tilde X_1,\dots, \tilde X_n$ are two samples 
each of size $n$ of i.i.d. copies of $X,$ $\hat \Sigma$ and $\tilde \Sigma$ are sample 
covariances based on $(X_1,\dots, X_n)$ and $(\tilde X_1,\dots, \tilde X_n),$
respectively, and $E:=\hat \Sigma-\Sigma,$ $\tilde E:=\tilde \Sigma-\Sigma.$

In what follows, we will use a concentration result for $\|\hat P_r-P_r\|_2^2-\|L_r(E)\|_2^2$ that was obtained 
in \cite{Koltchinskii_Lounici_HS} (see Theorem 5 there) and played a crucial role in the derivation of normal approximation bound 
of Theorem \ref{th:normal_approx}.

\begin{lemma}
\label{HS-conc-bd}
Suppose that for some $\gamma\in (0,1)$ condition (\ref{deltanleq}) holds. Then, for all $t\geq 1,$ with probability at least 
$1-e^{-t}$
\begin{align}
&
\nonumber
\Bigl|\|\hat P_r-P_r\|_2^2 - \|L_r(E)\|_2^2 - {\mathbb E}(\|\hat P_r-P_r\|_2^2 - \|L_r(E)\|_2^2)\Bigr|
\\
&
\lesssim_{\gamma} m_r \frac{\|\Sigma\|_{\infty}^3}{\bar g_r^3} \biggl(\frac{{\bf r}(\Sigma)}{n}\bigvee \frac{t}{n}\bigvee 
\biggl(\frac{t}{n}\biggr)^2\biggr)\sqrt{\frac{t}{n}}.
\end{align}
\end{lemma}

The first new result of this section is a useful representation for $(1+\hat b_r)^2- (1+b_r)^2$ 
that will be crucial in our proofs.

\begin{lemma}\label{lemma2}
Suppose for some $\gamma\in (\frac{23}{24},1)$ condition (\ref{deltanleq}) holds.
%Suppose that $\delta \in (0,1)$ such that $\frac{\delta}{\bar g_r} <\frac{1}{24}$.
Then, there exists a constant $D_1>0$ such that the following 
representation holds 
\begin{align}
\label{bias-interm2}
(1+ \hat b_r)^2 - (1+b_r)^2
&= \bigl \langle  L_r(E) , L_r(\tilde E)  \bigr \rangle - \frac{1}{2}\bigl( \|L_r(E)\|_2^2 - \mathbb E \|L_r(E)\|_2^2  \bigr)\notag\\
&  -\frac{1}{2}\bigl( \|L_r(\tilde E)\|_2^2 - \mathbb E \|L_r(\tilde E)\|_2^2  \bigr)+ \Upsilon_r
%&= v^\top \Xi_r + \Upsilon_r,
\end{align}
with the remainder term $\Upsilon_r$ that, for all $t\geq 1,$ with probability at least 
$1-e^{-t}$ satisfies the bound 
\begin{equation}
\label{bd_upsilon}
|\Upsilon_r| \leq D_{1}\frac{\|\Sigma\|_{\infty}^4}{\bar g_r^4} \biggl(\frac{{\bf r}(\Sigma)}{n}\bigvee 
\frac{t}{n} \bigvee \left( \frac{t}{n}\right)^3\biggr)\left( \sqrt{\frac{t}{n}} \bigvee \frac{t}{n} \right). 
\end{equation}
\end{lemma}

\begin{proof}
By the definition of $\hat b_r,$ we have
\begin{align}
\label{one_b_r}
\nonumber
(1+ \hat b_r)^2 &   = \bigl \langle \hat P_r, \tilde P_r \bigr \rangle\\
&
\nonumber
= \bigl \langle \hat P_r - \mathbb E \hat P_r, \tilde P_r  - \mathbb E \tilde P_r   \bigr \rangle + \bigl \langle \hat P_r - \mathbb E \hat P_r, P_r \bigr \rangle + \bigl \langle P_r, \tilde P_r  - \mathbb E \tilde P_r  \bigr\rangle  \\
&+ \bigl \langle \hat P_r - \mathbb E \hat P_r, \mathbb E \tilde P_r - P_r \bigr\rangle +\bigl \langle \mathbb E \hat P_r - P_r, \tilde P_r  - \mathbb E \tilde P_r  \bigr\rangle  +  \bigl \langle \mathbb E \hat P_r , \mathbb E \tilde P_r  \bigr\rangle.
\end{align}
In view of (\ref{Ehat}), we also have
\begin{align*}
\bigl \langle \mathbb E \hat P_r , \mathbb E \tilde P_r  \bigr\rangle &= \bigl \langle \mathbb (1+b_r) P_r + T_r , (1+b_r) P_r + T_r\bigr\rangle = (1+b_r)^2 + \|T_r\|_2^2,
\end{align*}
since $P_r$ and $T_r$ are orthogonal by definition of the latter.
%Let 
%$$
%\rho_r(\theta_r):= \langle (\hat P_r-(1+b_r)P_r) \theta_r,\theta_r\rangle, \ \
%\tilde \rho_r(\theta_r):= \langle (\tilde P_r-(1+b_r)P_r) \theta_r,\theta_r\rangle.
%$$
%Since $P_r$ and $T_r$ are orthogonal,
%$$
%\bigl \langle \hat P_r - \mathbb E \hat P_r, P_r \bigr \rangle  = \bigl \langle \hat P_r - (1+b_r)P_r -T_r,P_r 
%\bigr\rangle =  \bigl \langle \hat P_r - (1+b_r)P_r,P_r \bigr\rangle =  \rho_r(\theta_r),
%$$
%and similarly $\bigl\langle \tilde P_r  - \mathbb E \tilde P_r,  P_r  \bigr\rangle  =\tilde \rho_r(\theta_r).$
Thus, (\ref{one_b_r}) can be rewritten as 
\begin{align}
\label{bias-interm1}
(1+ \hat b_r)^2 - (1+b_r)^2
&= \bigl \langle \hat P_r - \mathbb E \hat P_r, \tilde P_r  - \mathbb E \tilde P_r   \bigr \rangle + 
\bigl \langle \hat P_r - \mathbb E \hat P_r, P_r \bigr \rangle + \bigl \langle P_r, \tilde P_r  - \mathbb E\tilde P_r  \bigr\rangle \notag \\
&+ \bigl \langle \hat P_r - \mathbb E \hat P_r, \mathbb E \tilde P_r - P_r \bigr\rangle +\bigl \langle \mathbb E \hat P_r - P_r, \tilde P_r - \mathbb E \tilde P_r  \bigr\rangle  +   \|T_r\|_2^2.
\end{align} 
Denote
$$
\hat \varrho_r := \bigl \langle \hat P_r - \mathbb E \hat P_r, P_r \bigr \rangle + \frac{1}{2}\bigl( \|L_r(E)\|_2^2 - \mathbb E \|L_r(E)\|_2^2\bigr),
\quad 
\tilde \varrho_r := \bigl \langle \tilde P_r - \mathbb E \tilde P_r, P_r \bigr \rangle+ \frac{1}{2}\bigl( \|L_r(\tilde E)\|_2^2 - \mathbb E \|L_r(\tilde E)\|_2^2  \bigr),
$$
where
$$
E:=\hat \Sigma-\Sigma, \ \ \tilde E:=\tilde \Sigma-\Sigma.
$$
We immediately get from (\ref{bias-interm1}) that
\begin{align*}
(1+ \hat b_r)^2 - (1+b_r)^2
& = \bigl \langle \hat P_r - \mathbb E \hat P_r, \tilde P_r  - \mathbb E \tilde P_r   \bigr \rangle - \frac{1}{2}\bigl( \|L_r(E)\|_2^2 - \mathbb E \|L_r(E)\|_2^2  \bigr) + \varrho_r\\
&-  \frac{1}{2}\bigl( \|L_r(\tilde E)\|_2^2 - \mathbb E \|L_r(\tilde E)\|_2^2  \bigr)+ \tilde\varrho_r+  \bigl \langle \hat P_r - \mathbb E \hat P_r, \mathbb E \tilde P_r - P_r \bigr\rangle\\
& +\bigl \langle \mathbb E \hat P_r - P_r, \tilde P_r  - \mathbb E \tilde P_r  \bigr\rangle  +   \|T_r\|_2^2.
\end{align*}
Since $\hat P_r  - \mathbb E \hat P_r = L_r(E) + R_r(E),$ $\tilde P_r  - \mathbb E \tilde P_r = L_r(\tilde E) + R_r(\tilde E),$
\begin{align*}
\bigl \langle \hat P_r - \mathbb E \hat P_r, \tilde P_r  - \mathbb E \tilde P_r   \bigr \rangle  &= \bigl \langle  L_r(E) + R_r(E), L_r(\tilde E) + R_r(\tilde E) \bigr \rangle\\
&=  \bigl \langle  L_r(E) , L_r(\tilde E)  \bigr \rangle +  \bigl \langle  L_r(E) , R_r(\tilde E) \bigr \rangle +  \bigl \langle  L_r(\tilde E),  R_r(E) \bigr \rangle +  \bigl \langle  R_r(E), R_r(\tilde E) \bigr \rangle.
\end{align*}
Combining the last two displays, we get that representation (\ref{bias-interm2}) holds 
%\begin{align}
%\label{bias-interm2}
%(1+ \hat b_r)^2 - (1+b_r)^2
%&= \bigl \langle  L_r(E) , L_r(\tilde E)  \bigr \rangle - \frac{1}{2}\bigl( \|L_r(E)\|_2^2 - \mathbb E \|L_r(E)\|%_2^2  \bigr)\notag\\
%&  -\frac{1}{2}\bigl( \|L_r(\tilde E)\|_2^2 - \mathbb E \|L_r(\tilde E)\|_2^2  \bigr)+ \Upsilon_r,
%&= v^\top \Xi_r + \Upsilon_r,
%\end{align}
with the remainder  
%we recall that $\Xi_r$ is defined in (\ref{Xi_r}), $v = \left(\frac{1}{\sqrt{2}},0,\frac{1}{2},\frac{1}{2},0\right)^\top$ and 
\begin{align*}
\Upsilon_r :=  \varrho_r + \tilde \varrho_r &+  
\bigl \langle \hat P_r - \mathbb E \hat P_r, \mathbb E \tilde P_r - P_r \bigr\rangle +\bigl \langle \mathbb E \hat P_r - P_r, \tilde P_r  - \mathbb E \tilde P_r  \bigr\rangle  +   \|T_r\|_2^2\\
& + \bigl \langle  L_r(E) , R_r(\tilde E) \bigr \rangle +  \bigl \langle  L_r(\tilde E),  R_r(E) \bigr \rangle +  \bigl \langle  R_r(E), R_r(\tilde E) \bigr \rangle.
\end{align*}
It remains to check that $\Upsilon_r$ satisfies bound (\ref{bd_upsilon}).

In what follows, we frequently use bounds of Theorems \ref{th_operator} and \ref{spectrum_sharper} along with bound (\ref{sha_sha}). Under condition (\ref{deltanleq}), we have
$$
\|\Sigma\|_{\infty} \left( \sqrt{\frac{\mathbf{r}(\Sigma)}{n}}\bigvee  \frac{\mathbf{r}(\Sigma)}{n}  \right) \lesssim \frac{\bar g_r}{2} \leq \frac{\|\Sigma\|_{\infty} }{2}.
$$
This implies that $\frac{\mathbf{r}(\Sigma)}{n} \lesssim 1$ and $\frac{\mathbf{r}(\Sigma)}{n} \lesssim \sqrt{\frac{\mathbf{r}(\Sigma)}{n}}$. Thus, the term $\frac{\mathbf{r}(\Sigma)}{n}$ in bounds of Theorems \ref{th_operator} and \ref{spectrum_sharper} and (\ref{sha_sha}) could be dropped. This is done in what follows without further notice.

Our next goal is to provide a bound on the remainder term $\Upsilon_r$ which can be done 
for an arbitrary multiplicity $m_r$ of $\mu_r.$ To this end, first note that bound (\ref{remaind})
easily implies that for any symmetric operator $B$ of finite rank $m$ the following bound holds 
with probability at least $1-e^{-t}:$
\begin{equation}
\label{R_r_bound}
\Bigl|\langle R_r(E), B\rangle\Bigr| \leq  
D_\gamma m \|B\|_{\infty}\frac{\|\Sigma\|_{\infty}^2}{\bar g_r^2}
\biggl(\sqrt{\frac{{\bf r}(\Sigma)}{n}}\bigvee \sqrt{\frac{t+\log(m)}{n}}\bigvee \frac{t+\log (m)}{n}\biggr)
\sqrt{\frac{t+\log (m)}{n}}.
\end{equation}
Indeed, it is enough to use the spectral representation $B=\sum_{j=1}^m \lambda_j (\phi_j\otimes \phi_j)$
of $B$ with eigenvalues $\lambda_j$ and orthonormal eigenvectors $\phi_j,$ to write 
$$
\Bigl|\langle R_r(E), B\rangle\Bigr| \leq \sum_{j=1}^m |\lambda_j|\Bigl|\langle R_r(E)\phi_j, \phi_j\rangle\Big|
\leq m\|B\|_{\infty} \max_{1\leq j\leq m}\Bigl|\langle R_r(E)\phi_j, \phi_j\rangle\Big|,
$$
to use bound (\ref{remaind}) with $t+\log (m)$ instead of $t$
in order to control bilinear forms $\Bigl|\langle R_r(E)\phi_j, \phi_j\rangle\Big|$
and, finally, to use the union bound. 

We will use bound (\ref{R_r_bound}) to control the last three terms 
in the expression for the remainder $\Upsilon_r.$ To control $\bigl \langle  L_r(\tilde E) , R_r(E) \bigr \rangle,$
we use (\ref{R_r_bound}) conditionally on $\tilde X_1,\dots, \tilde X_n$ with $B= L_r(\tilde E)$ 
(that is of rank at most $2m_r$) to get 
that with probability at least $1-e^{-t}$
\begin{align*}
&
\nonumber
\Bigl|\langle R_r(E), L_r(\tilde E)\rangle\Bigr| \lesssim
\\
&
 m_r \|L_r(\tilde E)\|_{\infty}\frac{\|\Sigma\|_{\infty}^2}{\bar g_r^2}
\biggl(\sqrt{\frac{{\bf r}(\Sigma)}{n}}\bigvee \sqrt{\frac{t+\log(2m_r)}{n}}\bigvee \frac{t+\log(2m_r)}{n}\biggr)
\sqrt{\frac{t+\log (2m_r)}{n}}.
\end{align*}
This should be combined with an upper bound on $\|L_r(\tilde E)\|_{\infty}$ that follows from (\ref{sha_sha})
and also holds with probability at least $1-e^{-t}:$
$$
\|L_r(\tilde E)\|_{\infty}\lesssim \frac{\|\tilde E\|_{\infty}}{\bar g_r}\lesssim 
\frac{\|\Sigma\|_{\infty}}{\bar g_r}
\left(\sqrt{\frac{\mathbf r(\Sigma) }{n}} \bigvee \sqrt{\frac{t}{n}}\bigvee \frac{t}{n} \right)
$$
(where it was also used that $\|C_r\|_{\infty}\leq \frac{1}{\bar g_r}$).
As a consequence, the following holds with probability at least $1-2e^{-t}:$
\begin{align}
\label{R_r_A}
&
\nonumber
\Bigl|\langle R_r(E), L_r(\tilde E)\rangle\Bigr| \lesssim_{\gamma}
\\
&
\nonumber
m_r \frac{\|\Sigma\|_{\infty}^3}{\bar g_r^3}
\left(\sqrt{\frac{\mathbf r(\Sigma) }{n}} \bigvee \sqrt{\frac{t}{n}} \bigvee \frac{t}{n}\right)
\\
&
\biggl(\sqrt{\frac{{\bf r}(\Sigma)}{n}}\bigvee \sqrt{\frac{t+\log(2m_r)}{n}}\bigvee \frac{t+\log(2m_r)}{n}\biggr)
\sqrt{\frac{t+\log (2m_r)}{n}}.
\end{align}
Of course, a similar bound also holds for $\Bigl|\langle L_r(E), R_r(\tilde E)\rangle\Bigr|.$ As to 
$\Bigl|\langle R_r(E), R_r(\tilde E)\rangle\Bigr|,$ observe that, by (\ref{remainder_A}),
(\ref{sha_sha}) and Theorem \ref{th_operator}, we have that with probability at least $1-e^{-t},$
$$
\|R_r(\tilde E)\|_{\infty} \leq \|S_r(\tilde E)\|_{\infty}+ {\mathbb E}\|S_r(\tilde E)\|_{\infty}
\lesssim  \frac{\|\tilde E\|_{\infty}^2}{\bar g_r^2} + \frac{{\mathbb E}\|\tilde E\|_{\infty}^2}{\bar g_r^2}
$$
$$
\lesssim \frac{\|\Sigma\|_{\infty}^2}{\bar g_r^2}\left(\sqrt{\frac{\mathbf r(\Sigma) }{n}} \bigvee \sqrt{\frac{t}{n}} \bigvee \frac{t}{n}\right)^2.
$$
Therefore, using again bound (\ref{remaind}) conditionally on $\tilde X_1,\dots, \tilde X_n$ 
with $B=R_r(\tilde E)$ we get that with probability $1-2e^{-t}$
\begin{align}
\label{R_r_B}
&
\nonumber
\Bigl|\langle R_r(E), R_r(\tilde E)\rangle\Bigr| \lesssim_{\gamma}
\\
&
\nonumber
m_r \frac{\|\Sigma\|_{\infty}^4}{\bar g_r^4}
\left(\sqrt{\frac{\mathbf r(\Sigma) }{n}} \bigvee \sqrt{\frac{t}{n}} \right)^2
\\
&
\biggl(\sqrt{\frac{{\bf r}(\Sigma)}{n}}\bigvee \sqrt{\frac{t+\log(2m_r)}{n}} \bigvee \frac{t+\log(2m_r)}{n}\biggr)
\sqrt{\frac{t+\log (2m_r)}{n}}.
\end{align}

To bound $\|T_r\|_2^2,$ note that
$$
\|T_r\|_2^2=\langle T_r,T_r\rangle \leq \|T_r\|_1\|T_r\|_{\infty},
$$
and, by the definition of $T_r,$ 
$$
\|T_r\|_1\leq \|{\mathbb E}\hat P_r-P_r\|_1+  \|P_r({\mathbb E}\hat P_r-P_r)P_r\|_1
\leq 2m_r {\mathbb E}\|\hat P_r-P_r\|_{\infty}+ m_r\|{\mathbb E}\hat P_r-P_r\|_{\infty}
\leq 3m_r {\mathbb E}\|\hat P_r-P_r\|_{\infty}.
$$
Using (\ref{bd_1}) and Theorem \ref{th_operator}, we get 
$$
\|T_r\|_1\lesssim_{\gamma} m_r \frac{\|\Sigma\|_{\infty}}{\bar g_r}\sqrt{\frac{{\bf r}(\Sigma)}{n}}.
$$
Therefore, by bound (\ref{T_r}),
\begin{equation}
\label{T_r_norm}
\|T_r\|_2^2 \leq \|T_r\|_1\|T_r\|_{\infty}
\lesssim_{\gamma}
m_r^2\frac{\|\Sigma\|_{\infty}^3}{\bar g_r^3}
\frac{{\bf r}(\Sigma)}{n}
\sqrt{\frac{1}{n}}.
\end{equation}

We will now control 
\begin{equation}
\label{thats}
\bigl \langle \hat P_r - \mathbb E \hat P_r, \mathbb E \tilde P_r - P_r \bigr\rangle
= 
\bigl \langle \hat P_r - \mathbb E \hat P_r, P_r W_rP_r\bigr\rangle
+
\bigl \langle \hat P_r - \mathbb E \hat P_r, T_r\bigr\rangle,
\end{equation}
where $W_r=\mathbb E \tilde P_r - P_r.$
Recall that $\hat P_r - \mathbb E \hat P_r=L_r(E)+R_r(E).$ 
Since $L_r(E)=P_rEC_r+C_rEP_r$ and $C_rP_r=P_rC_r=0,$ it is easy to see that  
$$
\langle L_r(E), P_r W_rP_r\rangle= \langle P_rEC_r, P_r W_rP_r\rangle
+\langle C_rEP_r, P_r W_rP_r\rangle=0.
$$ 
Thus, 
$$
\bigl \langle \hat P_r - \mathbb E \hat P_r, P_r W_rP_r\bigr\rangle = \langle R_r(E),P_r W_rP_r\rangle.
$$
Note that $B=P_rW_rP_r$ is an operator of rank at most $m_r$ and, in view of 
(\ref{bd_1}) and Theorem \ref{th_operator}, 
\begin{align*}
&
\|P_r W_rP_r\|_{\infty}\leq \|\mathbb E \tilde P_r - P_r\|_{\infty}
\leq {\mathbb E}\|\tilde P_r - P_r\|_{\infty}
\\
&
\lesssim \frac{\|\Sigma\|_{\infty}}{\bar g_r}\sqrt{\frac{{\bf r}(\Sigma)}{n}}.
\end{align*}
Thus, bound (\ref{R_r_bound})
implies that with probability at least $1-e^{-t}:$
\begin{align}
\label{thats_A}
&
\nonumber
\Bigl|\langle \hat P_r - \mathbb E \hat P_r, P_rW_rP_r\rangle\Bigr|=
\Bigl|\langle R_r(E), P_rW_rP_r\rangle\Bigr| 
\\
&
\lesssim_{\gamma}
m_r \frac{\|\Sigma\|_{\infty}^3}{\bar g_r^3}
\sqrt{\frac{{\bf r}(\Sigma)}{n}}
\biggl(\sqrt{\frac{{\bf r}(\Sigma)}{n}}\bigvee \sqrt{\frac{t+\log(m_r)}{n}}\bigvee \frac{t+\log (m_r)}{n}\biggr)
\sqrt{\frac{t+\log (m_r)}{n}}.
\end{align}
On the other hand, 
$$
\bigl| \langle \hat P_r - \mathbb E \hat P_r, T_r\rangle\bigr|
\leq \|\hat P_r - \mathbb E \hat P_r\|_1 \|T_r\|_{\infty}
\leq \Bigl(\|\hat P_r - P_r\|_1+ {\mathbb E}\|\hat P_r-P_r\|_1\Bigr)
\|T_r\|_{\infty}
$$
$$
\leq 2m_r\Bigl(\|\hat P_r - P_r\|_{\infty}+ {\mathbb E}\|\hat P_r-P_r\|_{\infty}\Bigr)
\|T_r\|_{\infty}.
$$
Using bounds (\ref{T_r}), (\ref{bd_1}), (\ref{sha_sha}) and Theorem \ref{th_operator}, we get 
\begin{align}
\label{thats_B}
\bigl| \langle \hat P_r - \mathbb E \hat P_r, T_r\rangle\bigr|
\lesssim_{\gamma}
m_r^2\frac{\|\Sigma\|_{\infty}^3}{\bar g_r^3}
\left(\sqrt{\frac{\mathbf r(\Sigma)}{n}} \bigvee \sqrt{\frac{t}{n}} \bigvee \frac{t}{n}\right)
\sqrt{\frac{{\bf r}(\Sigma)}{n}}\frac{1}{\sqrt{n}}.
\end{align}
It follows from (\ref{thats}) and bounds (\ref{thats_A}), (\ref{thats_B}) that with probability at least $1-2e^{-t}$
\begin{align}
\label{thats_C}
&
\nonumber
\Bigl|\bigl \langle \hat P_r - \mathbb E \hat P_r, \mathbb E \tilde P_r - P_r \bigr\rangle\Bigr|
\\
&
\lesssim_{\gamma}
m_r^2 \frac{\|\Sigma\|_{\infty}^3}{\bar g_r^3}
\sqrt{\frac{{\bf r}(\Sigma)}{n}}
\biggl(\sqrt{\frac{{\bf r}(\Sigma)}{n}}\bigvee \sqrt{\frac{t+\log(m_r)}{n}}\bigvee 
\frac{t+\log(m_r)}{n}\biggr)
\sqrt{\frac{t+\log (m_r)}{n}}.
\end{align}
Of course, the term $\bigl \langle \tilde P_r - \mathbb E \tilde P_r, \mathbb E \hat P_r - P_r \bigr\rangle$
can be bounded similarly. 

It remains to control $\varrho_r$ and $\tilde \varrho_r$. 
%By symmetry, it is sufficient to get a bound on $\varrho_r$. We will prove using Gaussian concentration that
%$$
%\varrho_r = O_{\mathbb P}\left(  \frac{\mathbf{r}(\Sigma)}{n^{3/2}} \right).
%$$
Note that $\langle L_r(E), P_r\rangle=0,$ implying that 
$$
\bigl \langle \hat P_r - \mathbb E \hat P_r, P_r \bigr \rangle =
\bigl \langle L_r(E)+S_r(E)-{\mathbb E} S_r(E), P_r \bigr \rangle
=
\bigl \langle S_r(E)-{\mathbb E} S_r(E), P_r \bigr \rangle.
$$
Therefore,
$$
\varrho_{r}=
\bigl\langle S_r(E), P_r \bigr\rangle   + \frac{1}{2}\|L_r(E)\|_2^2  - \mathbb E\left( \bigl \langle  S_r(E), P_r \bigr\rangle + \frac{1}{2}\|L_r(E)\|_2^2\right).
$$

The following lemma provides a concentration inequality for the random variable $\bigl\langle S_r(E), P_r \bigr\rangle   + \frac{1}{2}\|L_r(E)\|_2^2$ around its expectation (thus, implying a bound on $\varrho_{r}$).

\begin{lemma}
\label{le:conc}
Suppose that condition (\ref{deltanleq}) holds for some $\gamma \in (\frac{23}{24},1)$. 
%Suppose that $\delta\in (0,1/3]$ and $\|C_r\|_{\infty}\delta\leq 1/24.$ 
Then, there exists a constant $L>0$ such that for all $t\geq 1$
the following bound holds with probability at least $1-e^{-t}:$
\begin{align}
\label{conc_b}
&
\nonumber
\biggl| \bigl\langle S_r(E), P_r \bigr\rangle   + \frac{1}{2}\|L_r(E)\|_2^2  - \mathbb E\left( \bigl \langle  S_r(E), P_r \bigr\rangle + \frac{1}{2}\|L_r(E)\|_2^2\right)\biggr|
\\
&
\leq L m_r \frac{\|\Sigma\|_{\infty}^3}{\bar g_r^3}
\biggl(\frac{{\bf r}(\Sigma)}{n}\bigvee \frac{t}{n}\bigvee \biggl(\frac{t}{n}\biggr)^2\biggr)\left( \sqrt{\frac{t}{n}}\bigvee \frac{t}{n}\right).
\end{align}
\end{lemma}

Combining bounds (\ref{R_r_A}), (\ref{R_r_B}), (\ref{T_r_norm}), (\ref{thats_C}) and (\ref{conc_b}),
it is easy to derive the following bound on $\Upsilon_r$ that holds with probability at least $1-12e^{-t}:$
$$
|\Upsilon_r| \lesssim
m_r^2\frac{\|\Sigma\|_{\infty}^4}{\bar g_r^4} \biggl(\frac{{\bf r}(\Sigma)}{n}\bigvee 
\frac{t+\log(2m_r)}{n}\bigvee \biggl(\frac{t+\log(2m_r)}{n}\biggr)^3\biggr)\left(\sqrt{\frac{t+\log(2m_r)}{n}}\bigvee \frac{t+\log(2m_r)}{n}\right).
$$
The probability bound can be written as $1-e^{-t}$ by adjusting the constant in the inequality $\lesssim.$
For $m_r=1,$ this yields bound (\ref{bd_upsilon}) completing the proof of Lemma \ref{lemma2}. 
\qed

\end{proof}

We now prove Lemma \ref{le:conc}. To this end, we will use the following representations for operators $S_r(E).$
Given $L\subset \{1,\dots, k+1\},$ denote $m_L:={\rm card}(L)$ and 
$$
J_L:=\{\vec{j}:=(j_1,\dots,j_k,j_{k+1}): j_s=r, s\in L,  j_s\neq r, s\not\in L\}.
$$ 
Denote by $V_L$ the set of vectors $\nu=(\nu_l:l\in L^c)$ with nonnegative integer 
components such that $\sum_{l\in L^c}\nu_l=m_L-1.$ 
Finally, denote by ${\cal L}_k$ the set of all $L\subset \{1,\dots, k+1\}$ such 
that $L\neq \emptyset, L^c\neq \emptyset.$

\begin{lemma}\label{lem:pert-representations}
For all $r\geq 1,$
%\begin{equation}
%\label{perturb_ser_AA}
%\hat P_r-P_r =
%\sum_{k\geq 1}\sum_{L\subset {\cal L}_k}
%(-1)^{m_L-1}\sum_{\nu\in V_L} A_{\nu}(E)
%\end{equation}
%and 
\begin{equation}
\label{perturb_ser}
S_r(E)= 
\sum_{k\geq 2}\sum_{L\in \mathcal L_k}
(-1)^{m_L-1}\sum_{\nu\in V_L} A_{\nu}(E),
\end{equation}
where 
$$A_{\nu}(E):=B_1 E\dots B_k E B_{k+1}$$ 
with $B_l=P_r, l\in L$ and 
$B_l=C_r^{\nu_l+1}, l\in L^c.$ 
\end{lemma}

\begin{proof} 
It follows from the proof of Lemma 1 in \cite{Koltchinskii_Lounici_bilinear} that the following 
representation holds for $S_r(E):$
%$$
%\hat P_r-P_r=
%- \sum_{k\geq 1}\frac{1}{2\pi i} \oint_{\gamma_r}(-1)^{k} [R_{\Sigma}(\eta)E]^k R_{\Sigma}(\eta)d\eta 
%$$
%$$
%=-\sum_{k\geq 1}\frac{1}{2\pi i} \oint_{\gamma_r}(-1)^{k}
%\biggl[\sum_{j\geq 1}\frac{1}{\mu_j-\eta}P_jE\biggr]^k \sum_{j\geq 1}\frac{1}{\mu_j-\eta}P_jd\eta 
%$$
%$$
%=\sum_{k\geq 1}\sum_{j_1,\dots, j_k,j_{k+1}\geq 1}\frac{1}{2\pi i}
%\oint_{\gamma_r}\frac{d\eta}{\prod_{l=1}^{k+1}(\eta-\mu_{j_l})}
%P_{j_1}E\dots P_{j_k}EP_{j_{k+1}}.
%$$
%and similarly  
$$
S_r(E)=
- \sum_{k\geq 2}\frac{1}{2\pi i} \oint_{\gamma_r}(-1)^{k}[R_{\Sigma}(\eta)E]^kR_{\Sigma}(\eta)d\eta, 
$$
where $\gamma_r$ denotes the circle centered at $\mu_r$ of radius $\bar g_r/2$ with counterclockwise 
orientation and
$$
R_{\Sigma}(\eta)=(\Sigma-\eta I)^{-1}=\sum_{j\geq 1}\frac{1}{\mu_j-\eta} P_j
$$ 
denotes the resolvent of $\Sigma.$\footnote{In the case when $0$ is an eigenvalue of $\Sigma,$ the sum in the right hand side of the above formula extends to $j=\infty$ with $\mu_{\infty}=0.$ See also the 
remark after Lemma \ref{lem-pert-spectral}} 
Note also that the series in the above representation 
of $S_r(E)$ converges in the operator norm provided that $\|E\|_{\infty}<\frac{\bar g_r}{4}.$
It follows that
$$
S_r(E)=-\sum_{k\geq 2}\frac{1}{2\pi i} \oint_{\gamma_r}(-1)^{k}
\biggl[\sum_{j\geq 1}\frac{1}{\mu_j-\eta}P_jE\biggr]^k \sum_{j\geq 1}\frac{1}{\mu_j-\eta}P_jd\eta 
$$
$$
=\sum_{k\geq 2}
\sum_{j_1,\dots, j_k,j_{k+1}\geq 1}\frac{1}{2\pi i}
\oint_{\gamma_r}\frac{d\eta}{\prod_{l=1}^{k+1}(\eta-\mu_{j_l})}
P_{j_1}E\dots P_{j_k}EP_{j_{k+1}}.
$$
We have 
$$
\sum_{j_1,\dots, j_k,j_{k+1}\geq 1}\frac{1}{2\pi i}
\oint_{\gamma_r}\frac{d\eta}{\prod_{l=1}^{k+1}(\eta-\mu_{j_l})}
P_{j_1}E\dots P_{j_k}EP_{j_{k+1}}
$$
$$
=\sum_{L\subset \{1,\dots,k+1\}}\sum_{\vec{j}\in J_L}
\frac{1}{2\pi i}
\oint_{\gamma_r}\frac{d\eta}{(\eta-\mu_r)^{m_L}\prod_{l\in L^c}(\eta-\mu_{j_l})}
P_{j_1}E\dots P_{j_k}EP_{j_{k+1}}.
$$
Using Cauchy differentiation formula, we get 
$$
\frac{1}{2\pi i}
\oint_{\gamma_r}\frac{d\eta}{(\eta-\mu_r)^{m_L}\prod_{l\in L^c}(\eta-\mu_{j_l})}
=\frac{1}{(m_L-1)!}\biggl(\prod_{l\in L^c}(\eta-\mu_{j_l})^{-1}\biggr)^{(m_L-1)}_{|\eta=\mu_r}.
$$
In the cases when $L=\emptyset$ or $L^c=\emptyset$ the integral in the left hand side is equal to $0$. 
By generalized Leibniz rule,
$$
\biggl(\prod_{j\in L^c}(\eta-\mu_{j_l})^{-1}\biggr)^{(m_L-1)}_{|\eta=\mu_r}=
\sum_{\nu\in V_L} \frac{(m_L-1)!}{\prod_{l\in L^c}\nu_l!} \prod_{l\in L^c} (-1)^{\nu_l} \nu_l!(\mu_r-\mu_{j_l})^{-\nu_l-1}.
$$
Thus,
$$
\sum_{j_1,\dots, j_k,j_{k+1}\geq 1}\frac{1}{2\pi i}
\oint_{\gamma_r}\frac{d\eta}{\prod_{l=1}^{k+1}(\eta-\mu_{j_l})}
P_{j_1}E\dots P_{j_k}EP_{j_{k+1}}
$$
$$
\sum_{L\subset {\cal L}_k}\sum_{\vec{j}\in J_L}(-1)^{m_L-1}
\sum_{\nu\in V_L} \prod_{l\in L^c}(\mu_r-\mu_{j_l})^{-\nu_l-1}
P_{j_1}E\dots P_{j_k}EP_{j_{k+1}}.
$$
Given $\nu\in V_L,$ recall that  
$A_{\nu}(E)=B_1 E\dots B_k E B_{k+1},$ where $B_l=P_r, l\in L$ and 
$B_l=C_r^{\nu_l+1}, l\in L^c.$ 
It is easy to see that 
$$
\sum_{\vec{j}\in J_L} \prod_{l\in L^c}(\mu_r-\mu_{j_l})^{-\nu_l-1}
P_{j_1}E\dots P_{j_k}EP_{j_{k+1}}=A_{\nu}(E).
$$
Therefore,
$$
\sum_{j_1,\dots, j_k,j_{k+1}\geq 1}\frac{1}{2\pi i}
\oint_{\gamma_r}\frac{d\eta}{\prod_{l=1}^{k+1}(\eta-\mu_{j_l})}
P_{j_1}E\dots P_{j_k}EP_{j_{k+1}}
=
\sum_{L\in {\cal L}_k}
(-1)^{m_L-1}\sum_{\nu\in V_L} A_{\nu}(E)
$$
and 
%(\ref{perturb_ser_AA}), 
(\ref{perturb_ser}) follows. 
\qed
\end{proof}

\begin{remark}
By a simple combinatorics, 
\begin{equation}
\label{card_L}
{\rm card}\biggl(\bigcup_{L\subset \{1,\dots,k+1\}} V_L\biggr)\leq \sum_{m=0}^{k+1} {k+1\choose m}^2
={2(k+1)\choose k+1}\leq 2^{2(k+1)}.
\end{equation}
\end{remark}

It is easy to check that 
$$
\sum_{L\in \mathcal L_2}(-1)^{m_L-1}\sum_{\nu\in V_L} A_{\nu}(E)
$$
$$
=P_rEC_rEC_r + C_rEP_rEC_r+C_rEC_rEP_r - P_rEP_rEC_r^2
-P_r EC_r^2 EP_r - C_r^2EP_rEP_r.
$$
Using the fact that $C_rP_r =P_rC_r=0,$ this easily implies that 
$$
\sum_{L\in \mathcal L_2}(-1)^{m_L-1}\sum_{\nu\in V_L} \langle  A_{\nu}(E) , P_r\rangle 
= -{\rm tr}(P_rEC_r^2EP_r) = -\|P_rEC_r\|_2^2=
-\frac{1}{2}\|L_r(E)\|_2^2.
$$
Thus, we get
$$
\langle S_r(E), P_r \rangle +\frac{1}{2}\|L_r(E)\|_2^2 = \sum_{k\geq 3}\sum_{L\in \mathcal L_k}(-1)^{m_L-1}\sum_{\nu\in V_L} 
\langle A_{\nu}(E), P_r\rangle.
$$

%First, we recall that
%\begin{align*}
%S_r(E) & = -\frac{1}{2\pi i}\oint_{\gamma}\sum_{k\geq 2} [ R_{\Sigma}(\eta)E]^k R_{\Sigma}(\eta)d\eta\\
%&= -\frac{1}{2\pi i}\oint_{\gamma}[ R_{\Sigma}(\eta)E]^2 R_{\Sigma}(\eta)d\eta -\frac{1}{2\pi i}\oint_{\gamma}\sum_{k\geq 3} [ R_{\Sigma}(\eta)E]^k R_{\Sigma}(\eta)d\eta,
%\end{align*}
%and the following representation
%$$
%-\frac{1}{2\pi i} \oint_{\gamma} \bigl \langle [ R_{\Sigma}(\eta)E]^2  R_{\Sigma}(\eta) , P_r \bigr \rangle d\eta =- \|P_r E C_r\|_2^2=-\frac{1}{2} \|L_r(E)\|_2^2,
%$$
%which follows from Cauchy formula
%\begin{align*}
%-\frac{1}{2\pi i}\oint_{\gamma} [ R_{\Sigma}(\eta)E]^2 R_{\Sigma}(\eta)d\eta &= P_r E C_r E C_r + C_r E C_r E P_r + C_r E P_r E C_r\\
%&-  P_r E P_r E C_r^2 -  C_r^2 E P_r E P_r - P_r E C_r^2 E P_r.
%\end{align*}
%Thus, we get
%\begin{align*}
%\bigl\langle S_r(E), P_r \bigr\rangle   + \frac{1}{2}\|L_r(E)\|_2^2
% &= -\frac{1}{2\pi i} \oint_{\gamma}\sum_{k\geq 3} \bigl \langle [ R_{\Sigma}(\eta)E]^k  R_{\Sigma}(\eta) , P_r \bigr \rangle d\eta.
%\end{align*}

The next step is to study the concentration of the random variable
$\langle S_r(E), P_r \rangle +\frac{1}{2}\|L_r(E)\|_2^2$ around its expectation. More precisely,
we study the concentration of its ``truncated version"
\begin{align*}
\left(\langle S_r(E), P_r \rangle +\frac{1}{2}\|L_r(E)\|_2^2 \right) \varphi\Bigl(\frac{\|E\|_{\infty}}{\delta}\Bigr),
\end{align*}
where $\varphi$ is a Lipschitz function with constant $1$ on ${\mathbb R}_{+},$ $0\leq \varphi (s)\leq 1,$
$\varphi (s)=1, s\leq 1,$ $\varphi(s)=0, s>2.$
The value of $\delta>0$ will be chosen below in such a way that $\|E\|_{\infty}\leq \delta$
with a high probability.

The main ingredient of the proof is the classical Gaussian isoperimetric inequality that 
easily implies the following statement.

\begin{lemma}
\label{Gaussian_concentration}
Let $X_1,\dots, X_n$ be i.i.d. centered Gaussian random variables 
in ${\mathbb H}$ with covariance operator $\Sigma.$
Let $f:{\mathbb H}^n\mapsto {\mathbb R}$ be a function satisfying 
the following Lipschitz condition with some $L>0:$
$$
\Bigl|f(x_1,\dots, x_n)-f(x_1',\dots, x_n')\Bigr|\leq 
L\biggl(\sum_{j=1}^n \|x_j-x_j'\|^2\biggr)^{1/2},\ x_1,\dots, x_n,x_1',\dots,x_n'\in {\mathbb H}.
$$ 
Suppose that, for a real number $M$,
$$
\mathbb{P} \left\lbrace f(X_1,\ldots,X_n) \geq M  \right\rbrace \geq \frac{1}{4}\;\text{and}\; \mathbb{P} \left\lbrace f(X_1,\ldots,X_n) \leq M  \right\rbrace \geq \frac{1}{4}
$$
Then, there exists a numerical constant $D>0$ such that for all $t\geq 1$ 
$$
{\mathbb P}\Bigl\{|f(X_1,\dots, X_n)-M|\geq DL\|\Sigma\|_{\infty}^{1/2}\sqrt{t}\Bigr\}\leq e^{-t}.
$$
\end{lemma}

We will use Lemma \ref{Gaussian_concentration} that will be applied to the function 
\begin{align*}
f(X_1,\dots, X_n)&:=
\left(\langle S_r(E), P_r \rangle +\frac{1}{2}\|L_r(E)\|_2^2 \right) \varphi\Bigl(\frac{\|E\|_{\infty}}{\delta}\Bigr)\\
&= \sum_{k\geq 3}\sum_{L\in {\cal L}_k}
(-1)^{m_L-1}\sum_{\nu\in V_L} f_{\nu,L}(X_1,\dots, X_n),
\end{align*}
where 
$$
f_{\nu,L}(X_1,\dots, X_n):=\langle A_{\nu}(E),P_r\rangle \varphi\Bigl(\frac{\|E\|_{\infty}}{\delta}\Bigr),
$$
$$
E=\hat \Sigma-\Sigma, \ \ 
\hat \Sigma=n^{-1}\sum_{j=1}^n X_j\otimes X_j.
$$
With a little abuse of notation, assume for now that $X_1,\dots, X_n$
are nonrandom vectors in ${\mathbb H}.$
We now have to check the Lipschitz condition for the function $f.$

\begin{lemma}
\label{Lipschitz_constant}
%Let the conditions of Lemma \ref{le:conc} be satisfied.
Let $\delta>0$ and suppose that $\|C_r\|_{\infty}\delta\leq 1/24.$
Then, there exists a numerical constant $D>0$ such that, for all $X_1,\dots, X_n,
X_1',\dots, X_n'\in {\mathbb H},$
\begin{equation}
\label{lip_lip_CC}
|f(X_1,\dots, X_n)-f(X_1',\dots, X_n')| 
\leq 
D m_r \|C_r\|_{\infty}^3 \delta^2
\frac{\|\Sigma\|_{\infty}^{1/2}+\sqrt{\delta}}{\sqrt{n}}
\biggl(\sum_{j=1}^n \|X_j-X_j'\|^2\biggr)^{1/2}.
\end{equation}
\end{lemma}

\begin{proof}
Consider first each function $f_{\nu,L}$ separately.
Let $L\in {\cal L}_k$ for some $k\geq 3.$ 
Note that 
$$
f_{\nu,L}(X_1,\dots, X_n)=\langle B_1 E\dots B_k E B_{k+1},P_r\rangle \varphi\Bigl(\frac{\|E\|_{\infty}}{\delta}\Bigr),
$$ 
where $B_l=P_r, l\in L$ and $B_l=C_r^{\nu_l+1}, l\in L^c.$  
Therefore, we get 
\begin{align}
\label{fnuL}
&
|f_{\nu,L}(X_1,\dots, X_n)| \leq 
\|B_1\|_{\infty}\dots\|B_{k+1}\|_{\infty}\|E\|_{\infty}^{k} \|P_r\|_1 I(\|E\|_{\infty}\leq 2\delta)
\\
\nonumber
&
\leq \|B_1\|_{\infty}\dots\|B_{k+1}\|_{\infty} \|P_r\|_1(2\delta)^k.
\end{align}
For $X_1',\dots, X_n'\in {\mathbb H},$ denote 
$$
\hat \Sigma':=n^{-1}\sum_{j=1}^n X_j'\otimes X_j',\ \ 
E':=\hat \Sigma'-\Sigma.
$$
Then, we get
\begin{align}
&
\nonumber
|f_{\nu,L}(X_1,\dots, X_n)-f_{\nu,L}(X_1',\dots, X_n')|
\\
\nonumber
&
=\biggl|\langle B_1(E-E')B_2\dots EB_kEB_{k+1},P_r\rangle \varphi \biggl(\frac{\|E\|_{\infty}}{\delta}\biggr) 
\\
\nonumber
&
+\langle B_1E'B_2 (E-E')B_3\dots EB_kEB_{k+1},P_r\rangle \varphi \biggl(\frac{\|E\|_{\infty}}{\delta}\biggr)+ \dots
\\
\nonumber
&
+\langle B_1E'B_2\dots E'B_k(E-E')B_{k+1},P_r\rangle \varphi \biggl(\frac{\|E\|_{\infty}}{\delta}\biggr)
\\
\nonumber
&
+\langle B_1E'B_2\dots E'B_kE'B_{k+1},P_r\rangle 
\biggl(\varphi \biggl(\frac{\|E\|_{\infty}}{\delta}\biggr)-\varphi \biggl(\frac{\|E'\|_{\infty}}{\delta}\biggr)\biggr)\biggr|
\\
\nonumber
&
\leq k\|B_1\|_{\infty}\dots \|B_{k+1}\|_{\infty}\|P_r\|_1(\|E\|_{\infty}\vee \|E'\|_{\infty})^{k-1}
\|E-E'\|_{\infty}
\\
\nonumber
&
+
\|B_1\|_{\infty}\dots \|B_{k+1}\|_{\infty}\|P_r\|_1\|E'\|_{\infty}^{k}\frac{1}{\delta}\|E-E'\|_{\infty},
\end{align}
where we used the assumption that the Lipschitz constant of $\varphi$
is $1.$ By symmetry, $\|E'\|_{\infty}$ in the right hand side can be replaced by 
$\|E\|_{\infty}$ implying that
\begin{align}
\label{odin_odin}
&
|f_{\nu,L}(X_1,\dots, X_n)-f_{\nu,L}(X_1',\dots, X_n')|
\\
\nonumber
&
\leq k\|B_1\|_{\infty}\dots \|B_{k+1}\|_{\infty}\|P_r\|_1(\|E\|_{\infty}\vee \|E'\|_{\infty})^{k-1}
\|E-E'\|_{\infty}
\\
\nonumber
&
+
\|B_1\|_{\infty}\dots \|B_{k+1}\|_{\infty}\|P_r\|_1(\|E\|_{\infty}\wedge \|E'\|_{\infty})^{k}\frac{1}{\delta}\|E-E'\|_{\infty}.
\end{align}

If both $\|E\|_{\infty}\leq 2\delta$ and 
$\|E'\|_{\infty}\leq 2\delta,$ this implies the bound 
\begin{align}
\label{lip_lip}
&
|f_{\nu,L}(X_1,\dots, X_n)-f_{\nu,L}(X_1',\dots, X_n')|\leq 
\\
\nonumber
&
\|B_1\|_{\infty}\dots \|B_{k+1}\|_{\infty}\|P_r\|_1 (k+2)(2\delta)^{k-1} \|E-E'\|_{\infty}.
\end{align}
If $\|E\|_{\infty}\leq 2\delta,$ but $\|E'\|_{\infty}>2\delta,$ then 
$f_{\nu,L}(X_1',\dots,X_n')=0$ and, by (\ref{fnuL}), 
$$
|f_{\nu,L}(X_1,\dots, X_n)-f_{\nu,L}(X_1',\dots, X_n')|=|f_{\nu,L}(X_1,\dots, X_n)|
$$
$$
\leq \|B_1\|_{\infty}\dots \|B_{k+1}\|_{\infty} \|P_r\|_1 (2\delta)^{k}.
$$
If, in addition, $\|E-E'\|_{\infty}>\delta,$ then bound (\ref{lip_lip})
still holds. On the other hand, if $\|E-E'\|_{\infty}\leq \delta,$ then
$\|E'\|_{\infty}\leq 3\delta$ and we get a slightly worse bound than (\ref{lip_lip}):
\begin{align}
\label{lip_lip'}
&
|f_{\nu,L}(X_1,\dots, X_n)-f_{\nu,L}(X_1',\dots, X_n')|\leq 
\\
\nonumber
&
\|B_1\|_{\infty}\dots \|B_{k+1}\|_{\infty} \|P_r\|_1 (k+2)(3\delta)^{k-1} \|E-E'\|_{\infty}.
\end{align}
The case when $\|E\|_{\infty}> 2\delta$ and $\|E'\|_{\infty}\leq 2\delta$
can be handled similarly and the case when both $\|E\|_{\infty}> 2\delta$ and $\|E'\|_{\infty}> 2\delta$ is trivial since function $f_{\nu,L}$ becomes $0.$ In each of these cases, bound (\ref{lip_lip'})
holds.

The following bound (see Lemma 5 in \cite{Koltchinskii_Lounici_bilinear}) provides a control of $\|E - E'\|_{\infty}:$
\begin{equation}
\label{E__E'}
\|E-E'\|_{\infty}\leq 
\frac{4\|\Sigma\|_{\infty}^{1/2}+4\sqrt{2\delta}}{\sqrt{n}}
\biggl(\sum_{j=1}^n \|X_j-X_j'\|^2\biggr)^{1/2}
\bigvee
\frac{4}{n}
\sum_{j=1}^n \|X_j-X_j'\|^2.
\end{equation}
Substituting the last bound into (\ref{lip_lip'}), we get 
\begin{align}
\label{lip_lip''}
&
|f_{\nu,L}(X_1,\dots, X_n)-f_{\nu,L}(X_1',\dots, X_n')|\leq 
\\
\nonumber
&
\biggl(4\|B_1\|_{\infty}\dots \|B_{k+1}\|_{\infty}\|P_r\|_1 (k+2)(3\delta)^{k-1} 
\frac{\|\Sigma\|_{\infty}^{1/2}+\sqrt{2\delta}}{\sqrt{n}}
\biggl(\sum_{j=1}^n \|X_j-X_j'\|^2\biggr)^{1/2}\biggr)
\bigvee
\\
\nonumber
&
\biggl(4\|B_1\|_{\infty}\dots \|B_{k+1}\|_{\infty} \|P_r\|_1 (k+2)(3\delta)^{k-1} \frac{1}{n}
\sum_{j=1}^n \|X_j-X_j'\|^2\biggr).
\end{align}
In view of (\ref{fnuL}), the left hand side is also bounded 
from above by 
$$
2\|B_1\|_{\infty}\dots \|B_{k+1}\|_{\infty}\|P_r\|_1 (2\delta)^{k},
$$
which allows one to get from (\ref{lip_lip''}) that 
\begin{align}
\label{lip_lip'''}
&
|f_{\nu,L}(X_1,\dots, X_n)-f_{\nu,L}(X_1',\dots, X_n')|\leq 
\\
\nonumber
&
\biggl(4\|B_1\|_{\infty}\dots \|B_{k+1}\|_{\infty}\|P_r\|_1 (k+2)(3\delta)^{k-1} 
\frac{\|\Sigma\|_{\infty}^{1/2}+\sqrt{2\delta}}{\sqrt{n}}
\biggl(\sum_{j=1}^n \|X_j-X_j'\|^2\biggr)^{1/2}\biggr)
\bigvee
\\
\nonumber
&
\biggl(4\|B_1\|_{\infty}\dots \|B_{k+1}\|_{\infty} \|P_r\|_1 (3\delta)^{k-1} \biggl((k+2)\frac{1}{n}
\sum_{j=1}^n \|X_j-X_j'\|^2\bigwedge \delta\biggr)\biggr).
\end{align}
In the case when 
$$
\biggl(\sum_{j=1}^n \|X_j-X_j'\|^2\biggr)^{1/2} \leq 
\sqrt{\frac{\delta n}{k+2}},
$$
we have 
$$
4\|B_1\|_{\infty}\dots \|B_{k+1}\|_{\infty} \|P_r\|_1 (3\delta)^{k-1} \biggl((k+2)\frac{1}{n}
\sum_{j=1}^n \|X_j-X_j'\|^2\bigwedge \delta\biggr)
$$
$$
\leq 4\|B_1\|_{\infty}\dots \|B_{k+1}\|_{\infty}\|P_r\|_1 (3\delta)^{k-1/2}
\sqrt{k+2}\frac{1}{\sqrt{n}}\biggl(\sum_{j=1}^n \|X_j-X_j'\|^2\biggr)^{1/2}.
$$
It is equally easy to check that the same bound holds in the opposite case.
As a consequence, (\ref{lip_lip'''}) implies that 
\begin{align}
\label{lip_lip_AA}
&
|f_{\nu,L}(X_1,\dots, X_n)-f_{\nu,L}(X_1',\dots, X_n')|\leq 
\\
\nonumber
&
4\|B_1\|_{\infty}\dots \|B_{k+1}\|_{\infty} \|P_r\|_1 (k+2)(3\delta)^{k-1} 
\frac{\|\Sigma\|_{\infty}^{1/2}+\sqrt{2\delta}}{\sqrt{n}}
\biggl(\sum_{j=1}^n \|X_j-X_j'\|^2\biggr)^{1/2}.
\end{align}

Note that 
\begin{equation}
\label{B_C}
\|B_1\|_{\infty}\dots \|B_{k+1}\|_{\infty}=\prod_{l\in L^c}\|C_r^{\nu_l+1}\|_{\infty}
\leq \|C_r\|_{\infty}^{\sum_{l\in L^c}(\nu_l+1)}
=\|C_r\|_{\infty}^{k},
\end{equation}
where we used the facts that 
$$
\sum_{l\in L^c}\nu_l=m_L-1,\ \ {\rm card}(L^c)=k+1-m_L.
$$
Thus, we get from (\ref{lip_lip_AA}) 
\begin{align}
%\label{lip_lip_BB}
\nonumber
&
|f_{\nu,L}(X_1,\dots, X_n)-f_{\nu,L}(X_1',\dots, X_n')|\leq 
\\
\nonumber
&
4\|C_r\|_{\infty}^k \|P_r\|_1 (k+2)(3\delta)^{k-1} 
\frac{\|\Sigma\|_{\infty}^{1/2}+\sqrt{2\delta}}{\sqrt{n}}
\biggl(\sum_{j=1}^n \|X_j-X_j'\|^2\biggr)^{1/2},
\end{align}
which, taking also into account (\ref{card_L}), yields 
\begin{align}
\label{lip_lip_BB}
&
|f(X_1,\dots, X_n)-f(X_1',\dots, X_n')| 
\\
\nonumber
&
\leq 
4\sum_{k\geq 3}\sum_{L\in {\cal L}_k}\sum_{\nu\in V_L}
\|C_r\|_{\infty}^k \|P_r\|_1(k+2)(3\delta)^{k-1} 
\frac{\|\Sigma\|_{\infty}^{1/2}+\sqrt{2\delta}}{\sqrt{n}}
\biggl(\sum_{j=1}^n \|X_j-X_j'\|^2\biggr)^{1/2}
\\
\nonumber
&
\leq 
4\sum_{k\geq 3}2^{2(k+1)}\|C_r\|_{\infty}^k \|P_r\|_1 (k+2)(3\delta)^{k-1} 
\frac{\|\Sigma\|_{\infty}^{1/2}+\sqrt{2\delta}}{\sqrt{n}}
\biggl(\sum_{j=1}^n \|X_j-X_j'\|^2\biggr)^{1/2}
\\
\nonumber
&
\leq 
4\sum_{k\geq 3}(k+2)2^{2(k+1)}3^{k-1} \biggl(\frac{1}{24}\biggr)^{k-3} 
\|C_r\|^3 \|P_r\|_1 \delta^2
\frac{\|\Sigma\|_{\infty}^{1/2}+\sqrt{2\delta}}{\sqrt{n}}
\biggl(\sum_{j=1}^n \|X_j-X_j'\|^2\biggr)^{1/2}
\\
\nonumber
&
\leq 
D \|C_r\|^3 \|P_r\|_1 \delta^2
\frac{\|\Sigma\|_{\infty}^{1/2}+\sqrt{\delta}}{\sqrt{n}}
\biggl(\sum_{j=1}^n \|X_j-X_j'\|^2\biggr)^{1/2},
\end{align}
where $D$ is a numerical constant and 
we used the condition $\|C_r\|_{\infty}\delta\leq 1/24.$
\qed
\end{proof}

We return to the proof of Lemma \ref{le:conc}. 

\begin{proof}
Note that, under condition (\ref{deltanleq}),
the lower bound of Theorem \ref{th_operator} implies that ${\bf r}(\Sigma)\lesssim n.$ 
Let $t\geq 1$ and define
$$
\delta_n(t):=\E\|\hat\Sigma - \Sigma\|_{\infty } + C \|\Sigma\|_{\infty} \left[  \sqrt{\frac{t}{n}} \bigvee \frac{t}{n}  \right].
$$
If constant $C$ in the above definition is sufficiently large and ${\bf r}(\Sigma)\lesssim n,$ then it follows from Theorem \ref{spectrum_sharper}
that 
$\|E\|_{\infty}=\|\hat \Sigma-\Sigma\|_{\infty}\leq \delta_n(t)$
with probability at least $1-e^{-t}.$
Note also that, under condition (\ref{deltanleq}), 
$$
\delta_n(t) \lesssim \|\Sigma\|_{\infty} \left[ \sqrt{\frac{\mathbf{r}(\Sigma)}{n}} \bigvee \sqrt{\frac{t}{n}} \bigvee \frac{t}{n} \right] 
$$ 
(since ${\bf r}(\Sigma)\lesssim n$).

Assume that $ \delta_n(t)\leq  \frac{\bar g_r}{24}.$ Since $\bar g_r\leq \|\Sigma\|_{\infty},$
we have    
$$
C \left[ \sqrt{\frac{t}{n}}\bigvee  \frac{t}{n} \right] \leq \frac{\bar g_r}{24\|\Sigma\|_{\infty}} \leq 1,
$$
which implies that $t\lesssim n.$ 
Thus, in view of the upper bound of Theorem \ref{th_operator}, 
$$
\delta_n(t) \lesssim \|\Sigma\|_{\infty} \left[ \sqrt{\frac{\mathbf{r}(\Sigma)}{n}} \bigvee \sqrt{\frac{t}{n}}\right].
$$

For a random variable $\xi,$ denote by ${\rm Med}(\xi)$ its median. 
Let 
$$M := 
\mathrm{Med}\left(\langle S_r(E), P_r \rangle + \frac{1}{2}\|L_r(E)\|_2^2\right).
$$ 
In what follows, we set $\delta := \delta_n(t)$ in the definition of function $f(X_1,\dots, X_n).$
%Let $\delta := \delta_n(t)$ and 
Suppose that $t\geq \log(4)$ (by adjusting the values of the constants the resulting bound can be easily extended to $t\geq 1$
as it is claimed in Lemma \ref{le:conc}). 
Then, we have ${\mathbb P}\{\|\hat \Sigma-\Sigma\|_{\infty}\geq \delta_n(t)\}\leq \frac{1}{4},$ and
\begin{align*}
\mathbb{P}\left\{ f(X_1,\ldots,X_n) \geq M \right) &\geq \mathbb{P}\left\{f(X_1,\ldots,X_n) \geq M, \, \|E\|_{\infty}<\delta \right\}\\
&\geq \mathbb{P}\left\{ \langle S_r(E), P_r \rangle + \frac{1}{2}\|L_r(E)\|_2^2 \geq M \right\} - \mathbb{P}\left\{\|E\|_{\infty}\geq \delta\right\} \geq \frac{1}{4}.
\end{align*}
Quite similarly, $\mathbb{P}\left\{ f(X_1,\ldots,X_n) \leq M \right\} \geq \frac{1}{4}$. It follows from Lemma \ref{Gaussian_concentration} that with probability at least $1-e^{-t}$ 
\begin{align*}
\left|  f(X_1,\ldots, X_n) - M  \right| &\leq D m_r\frac{\delta^2}{\bar g_r^3} \|\Sigma\|_{\infty}^{1/2}\left(\|\Sigma\|_{\infty}^{1/2} + \sqrt{\delta}\right)\sqrt{\frac{t}{n}}\\
&\leq D'm_r \frac{\|\Sigma\|_{\infty}^{3}}{\bar g_r^3}\left( \frac{\mathbf{r}(\Sigma)}{n} \bigvee \frac{t}{n} \right)\sqrt{\frac{t}{n}},
\end{align*}
for some numerical constant $D'>0.$ Since on the event $\{\|E\|_{\infty}\leq \delta\}$ 
$$\langle S_r(E), P_r \rangle + \frac{1}{2}\|L_r(E)\|_2^2 =f(X_1,\dots, X_n),$$
we easily obtain that with probability at least $1-2e^{-t}$ 
\begin{align}
\label{interconc}
\left|  \langle S_r(E), P_r \rangle + \frac{1}{2}\|L_r(E)\|_2^2 - M  \right| \leq D'm_r \frac{\|\Sigma\|_{\infty}^{3}}{\bar g_r^3}\left( \frac{\mathbf{r}(\Sigma)}{n} \bigvee \frac{t}{n} \right)\sqrt{\frac{t}{n}}.
\end{align}

It remains to prove a similar bound in the case when  $\delta_n(t) > \frac{\bar g_r}{24}$. By definition of $\delta_n(t)$ and in view of assumption (\ref{deltanleq}), we get
\begin{align}
\label{interdelta}
C\|\Sigma\|_{\infty}\left( \sqrt{\frac{t}{n}}\bigvee \frac{t}{n}\right) > \frac{\bar g_r}{24} - \mathbb{E}\|\hat\Sigma - \Sigma\|_{\infty}\geq \frac{\bar g_r}{24} - \frac{\bar g_r}{48} =\frac{\bar g_r}{48}.
\end{align}

In view of (\ref{linear_perturb}), (\ref{remainder_A}), the fact that $\|P_r\|_1 = m_r$ 
and the trace duality inequality, we obtain
\begin{align*}
\left|\langle S_r(E), P_r \rangle + \frac{1}{2}\|L_r(E)\|_2^2\right| &\leq  \|S_r(E)\|_{\infty} \|P_r\|_1 + \|C_r E P_r\|_{2}^2\\
&\leq 14 \frac{\|E\|_{\infty}^2}{\bar g_r^2}\|P_r\|_1 + \|C_r\|_{\infty}^2 \|E\|_{\infty}^2  \|P_r\|_{1}\\
&\leq 15 m_r \frac{\|E\|_{\infty}^2}{\bar g_r^2}.
\end{align*}
Since $\mathbb{P}\{\|E\|_{\infty}\leq \delta\}\geq 1-e^{-t}$, we get that for all $t\geq 1$ with probability at least $1-e^{-t}$ that
\begin{align*}\label{interconc2}
\left|\langle S_r(E), P_r \rangle + \frac{1}{2}\|L_r(E)\|_2^2\right| &\leq D m_r \frac{\|\Sigma\|_{\infty}^2}{\bar g_r^2} \left( \frac{\mathbf{r}(\Sigma)}{n} \bigvee \frac{t}{n} \bigvee \left( \frac{t}{n}\right)^2 \right)\notag
%\\
%&\leq  D m_r \frac{\|\Sigma\|_{\infty}^3}{\bar g_r^3} \left( \frac{\mathbf{r}(\Sigma)}{n} \vee \frac{t}{n} \vee \left( \frac{t}{n}\right)^2 \right)%\left( \sqrt{\frac{t}{n}}\bigvee \frac{t}{n}\right),
\end{align*}
for some numerical constant $D>0.$ Using this bound with $t=\log 4,$ we easily get that 
$$
|M|\leq {\rm Med} \biggl(\biggl|\langle S_r(E), P_r \rangle + \frac{1}{2}\|L_r(E)\|_2^2\biggr|\biggr)
\leq 
D m_r \frac{\|\Sigma\|_{\infty}^2}{\bar g_r^2} \left( \frac{\mathbf{r}(\Sigma)}{n} \bigvee \frac{\log 4}{n} \bigvee \left( \frac{\log 4}{n}\right)^2 \right).
$$
Combining the last two displays, we get that for some constant $D>0$ and  for all $t\geq 1$ with probability at least $1-e^{-t}$ 
\begin{equation}
\label{interconc''}
\left|\langle S_r(E), P_r \rangle + \frac{1}{2}\|L_r(E)\|_2^2 - M \right| \leq  
D m_r \frac{\|\Sigma\|_{\infty}^2}{\bar g_r^2} 
\left( \frac{\mathbf{r}(\Sigma)}{n} \bigvee \frac{t}{n} \bigvee \left( \frac{t}{n}\right)^2 \right).
\end{equation}
If $\delta_n(t)> \frac{\bar g_r}{24},$ then (\ref{interdelta}) holds and it follows from bound (\ref{interconc''})
that with some constant $D>0$
\begin{equation}
\label{interconc_A}
\left|\langle S_r(E), P_r \rangle + \frac{1}{2}\|L_r(E)\|_2^2 - M \right| \leq  
D m_r \frac{\|\Sigma\|_{\infty}^3}{\bar g_r^3} 
\left( \frac{\mathbf{r}(\Sigma)}{n} \bigvee \frac{t}{n} \bigvee \left( \frac{t}{n}\right)^2 \right)
\left(\sqrt{\frac{t}{n}}\bigvee \frac{t}{n}\right).
\end{equation}
Of course, in the case when $\delta_n(t)\leq \frac{\bar g_r}{24},$ bound (\ref{interconc_A})
also holds (it follows from bound (\ref{interconc})). 
By integrating tail probabilities of bound (\ref{interconc_A}) that holds for all $t\geq 1$ we easily get 
$$
\left|\E\left[\langle S_r(E), P_r \rangle + \frac{1}{2}\|L_r(E)\|_2^2\right] - M\right| \leq  
\E \left|\left[\langle S_r(E), P_r \rangle + \frac{1}{2}\|L_r(E)\|_2^2\right] - M\right| \leq  
$$
$$
\leq D m_r \frac{\|\Sigma\|_{\infty}^3}{\bar g_r^3} \left( \frac{\mathbf{r}(\Sigma)}{n} \bigvee \frac{1}{n} \bigvee \left( \frac{1}{n}\right)^2 \right)\sqrt{\frac{1}{n}}
$$
for some $D>0.$ Thus, we can replace the median $M$ in bound (\ref{interconc_A}) by the expectation 
which yields the bound of Lemma \ref{le:conc}.
\qed
\end{proof}

Consider now three samples $(X_1,\dots, X_n),$ $(\tilde X_1,\dots, \tilde X_n)$ and $(\bar X_1,\dots, \bar X_n)$
of i.i.d. copies of $X$ with $\hat \Sigma, \tilde \Sigma$ and $\bar \Sigma$ being the sample covariances 
based on the corresponding samples of size $n.$ Let $E:=\hat \Sigma-\Sigma, \tilde E:= \tilde \Sigma-\Sigma$ 
and $\bar E:=\bar \Sigma-\Sigma.$ In view of the representation of Lemma \ref{lemma2}, to study the asymptotic 
behavior of $(1+\hat b_r)^2-(1+b_r)^2$ and other related statistics we will have to deal with random vectors 
\begin{align}\label{Xi_r}
\Xi_{r} := \left(\begin{array}{c}
\sqrt{2}\bigl \langle  L_r(E) , L_r(\tilde E)  \bigr \rangle \vspace{.1cm}
 \\
 \sqrt{2}\bigl \langle  L_r(\tilde E) , L_r(\bar E)  \bigr \rangle \vspace{.1cm}
 \\
\|L_r(E)\|_2^2 - \mathbb E \|L_r(E)\|_2^2 \vspace{.1cm}
\\
\|L_r(\tilde E)\|_2^2 - \mathbb E \|L_r(\tilde E)\|_2^2\vspace{.1cm}
\\
\|L_r(\bar E)\|_2^2 - \mathbb E \|L_r(\bar E)\|_2^2
\end{array}
\right)
=\left(\begin{array}{c}
 2\sqrt{2}\bigl \langle P_r E C_r ,  P_r \tilde E C_r  \bigr \rangle \vspace{.1cm}
 \\
 2\sqrt{2}\bigl \langle  P_r\tilde E C_r , P_r \bar E C_r  \bigr \rangle\vspace{.1cm}
 \\
2(\|P_r E C_r\|_2^2 - \mathbb E \|P_r E C_r\|_2^2) \vspace{.1cm}
\\
2(\|P_r \tilde E C_r\|_2^2 - \mathbb E \|P_r \tilde E C_r\|_2^2) \vspace{.1cm}
\\
2(\|P_r \bar E C_r\|_2^2 - \mathbb E \|P_r \bar E C_r\|_2^2)
\end{array}
\right).
\end{align}
%admits the following equivalent representation
%$$
%\Xi_{r} = 
%\left(\begin{array}{c}
% 2\sqrt{2}\bigl \langle P_r E C_r ,  P_r \tilde E C_r  \bigr \rangle \vspace{.1cm}
% \\
% 2\sqrt{2}\bigl \langle  P_r\tilde E C_r , P_r \bar E C_r  \bigr \rangle\vspace{.1cm}
% \\
%2(\|P_r E C_r\|_2^2 - \mathbb E \|P_r E C_r\|_2^2) \vspace{.1cm}
%\\
%2(\|P_r \tilde E C_r\|_2^2 - \mathbb E \|P_r \tilde E C_r\|_2^2) \vspace{.1cm}
%\\
%2(\|P_r \bar E C_r\|_2^2 - \mathbb E \|P_r \bar E C_r\|_2^2)
%\end{array}
%\right).
%$$

Let $\{\eta_{j,k}, \tilde \eta_{j,k},  \bar \eta_{j,k},\, k\in \Delta_r,\,j\in \Delta_s,\,  s\neq r\}$ be i.i.d. standard normal random variables. Define the random vector
\begin{align}\label{joint-technical-1}
\Theta_r :=  \left(\begin{array}{c}
\sqrt{2}\sum_{k \in \Delta_r} \sum_{s\neq r} \frac{\mu_s}{(\mu_s - \mu_r)^2} \sum_{j \in \Delta_s} \eta_{j,k}\tilde \eta_{j,k}  \vspace{.1cm}
 \\
 \sqrt{2}\sum_{k \in \Delta_r} \sum_{s\neq r} \frac{\mu_s}{(\mu_s - \mu_r)^2} \sum_{j \in \Delta_s} \tilde \eta_{j,k}\bar \eta_{j,k}  \vspace{.1cm}
 \\
\sum_{k \in \Delta_r} \sum_{s\neq r} \frac{\mu_s}{(\mu_s - \mu_r)^2} \sum_{j \in \Delta_s} (\eta_{j,k}^2-1)\vspace{.1cm}
\\
\sum_{k \in \Delta_r} \sum_{s\neq r} \frac{\mu_s}{(\mu_s - \mu_r)^2} \sum_{j \in \Delta_s} (\tilde \eta_{j,k}^2-1)
\\
\sum_{k \in \Delta_r} \sum_{s\neq r} \frac{\mu_s}{(\mu_s - \mu_r)^2} \sum_{j \in \Delta_s} (\bar \eta_{j,k}^2-1)
\end{array}
\right).
\end{align}

\begin{lemma}\label{lemma:joint-technical-2}
The following representation holds:
\begin{align}\label{joint-technical-2}
n\Xi_{r} = 2\mu_r \tilde \Theta_r+ \xi,
\end{align}
where $\tilde \Theta_r$ is a random vector in $\R^5$ whose distribution coincides with the distribution of $\Theta_r$
and the components $\xi_j$ of the remainder $\xi \in {\mathbb R}^5$ satisfy the following bound:
$$
\max_{1\leq j\le 5}{\mathbb E}|\xi_j|\lesssim \frac{\|\Sigma\|_{\infty}^2}{\bar g_r^2}\biggl(\frac{m_r^{5/2}}{\sqrt{n}} \vee \frac{m_r^3}{n}\biggr) 
{\bf r}(\Sigma).
$$
\end{lemma}

%We now prove that
%\begin{align}\label{joint-technical-2}
%n\,\Xi_{r} = 2\mu_r \left(1+ O_{\mathbb P}\left( n^{-1/2}\right)\right) \tilde\Theta_r,
%\end{align}
\begin{proof} Set
$$
U = \frac{1}{\sqrt{n}}\sum_{j=1}^n P_r X_j \otimes C_r X_j,\quad \tilde U = \frac{1}{\sqrt{n}}\sum_{j=1}^n P_r \tilde X_j \otimes C_r \tilde X_j,\quad \bar U = \frac{1}{\sqrt{n}}\sum_{j=1}^n P_r \bar X_j \otimes C_r \bar X_j
$$
and note that 
$$
n\Xi_r = 
\left(\begin{array}{c}
\vspace{0.1cm}  2\sqrt{2}\bigl\langle U, \tilde U \bigr\rangle\\
\vspace{0.1cm}  2\sqrt{2}\bigl\langle \tilde U, \bar U \bigr\rangle\\
\vspace{0.1cm}  2(\| U\|_2^2 - \mathbb E \| U\|_2^2 )\\
 2( \| \tilde U\|_2^2 - \mathbb E \| \tilde U\|_2^2 )\\ 
 2( \| \bar U\|_2^2 - \mathbb E \| \bar U\|_2^2 )
\end{array}\right).
$$
Let
$$
\Gamma_r = n^{-1}\sum_{j=1}^n P_r X_j \otimes P_r X_j, \quad \tilde \Gamma_r = n^{-1}\sum_{j=1}^n P_r \tilde X_j \otimes P_r \tilde X_j,\quad \bar \Gamma_r = n^{-1}\sum_{j=1}^n P_r \bar X_j \otimes P_r \bar X_j
$$
be the sample covariance operators based, respectively, 
on the ``projected" samples $P_r X_j, j=1,\dots ,n,$ $P_r \tilde X_j, j=1,\dots, n$ and $P_r \bar X_j, j=1,\dots, n$
of i.i.d. centered Gaussian random variables with covariance operator $P_r\Sigma P_r=\mu_rP_r.$
$\Gamma_r, \tilde \Gamma_r, \bar \Gamma_r$ can be viewed as symmetric positive semi-definite operators acting in the eigenspace of eigenvalue $\mu_r$ and they admit the following spectral decompositions:
$$
\Gamma_r = \sum_{k \in \Delta_r} \gamma_r \phi_k \otimes \phi_k,\quad \tilde \Gamma_r = \sum_{k\in \Delta_r} \tilde \gamma_r \tilde \phi_k \otimes \tilde \phi_k,\quad \bar \Gamma_r = \sum_{k\in \Delta_r} \bar \gamma_r \bar \phi_k \otimes \bar\phi_k,
$$
where $\gamma_k\geq 0$ are the eigenvalues of $\Gamma_r$ with associated eigenvectors $\phi_k$, $\tilde \gamma_r\geq 0$ are the eigenvalues of $\tilde \Gamma_r$ with associated eigenvectors $\tilde \phi_k$ and $\bar \gamma_r\geq 0$ are the eigenvalues of $\bar \Gamma_r$ with associated eigenvectors $\bar \phi_k$. 
Note also that $\{\phi_k,\, k \in \Delta_r\}$, $\{\tilde \phi_k,\, k \in \Delta_r\}$ and $\{\bar\phi_k,\, k \in \Delta_r\}$ are three possibly different orthonormal bases of the eigenspace of $\mu_r$.

Let $X^{(k)}$, $\tilde X^{(k)}$, $\bar X^{(k)}$, $k\in \Delta_r$ be independent copies of $X$ (also independent of $X_j$, $\tilde X_j$, $\bar X_j$, $j=1,\ldots,n$). Denote
$$
V = \sum_{k\in \Delta_r} \sqrt{\gamma_k} \phi_k \otimes C_r X^{(k)} ,\quad \tilde V = \sum_{k\in \Delta_r} \sqrt{\tilde \gamma_k}  \tilde \phi_k \otimes C_r \tilde X^{(k)},\quad \bar V = \sum_{k\in \Delta_r} \sqrt{\bar \gamma_k}  \bar\phi_k \otimes C_r \bar X^{(k)}
$$
Given $\{P_r X_1,\ldots, P_r X_n, P_r \tilde X_1,\ldots, P_r\tilde X_n, P_r \bar X_1,\ldots, P_r\bar X_n\}$, the conditional distributions of $(U,\tilde U, \bar U)$ and $(V,\tilde V, \bar V)$ are the same. To see this note that, conditionally on $\{P_r X_1,\ldots, P_r X_n, P_r \tilde X_1,\ldots, P_r\tilde X_n, P_r \bar X_1,\ldots, P_r\bar X_n\},$ $U, \tilde U, \bar U$ are independent centered Gaussian 
random operators and so are $V,\tilde V,\bar V.$\footnote{Recall that $P_r X_j$ and $C_r X_j$ are independent since they are jointly Gaussian and uncorrelated (the last property follows from 
the fact that $P_r C_r=C_r P_r=0$). Thus, conditionally on $\{P_r X_j\},$ $U$ is a mean zero Gaussian
random operator.}
Thus, it is enough to check that conditionally on the same random variables 
the covariance operators of $U$ and $V$ coincide (of course, the same would apply 
to the couples $\tilde U$ and $\tilde V,$ $\bar U$ and $\bar V$). To this end, 
let $T$ denote a linear mapping from ${\mathbb H}\otimes {\mathbb H}\otimes {\mathbb H}\otimes {\mathbb H}$ into itself
such that 
$$
T(u_1\otimes u_2\otimes u_3\otimes u_4)= (u_1\otimes u_3\otimes u_2\otimes u_4)
$$
(note that $T$ is uniquely defined). By an easy computation,
$$
{\mathbb E}(U\otimes U| P_rX_j, j=1,\dots, n)= T(\Gamma_r\otimes (C_r\Sigma C_r))
={\mathbb E}(V\otimes V| P_rX_j, j=1,\dots, n),
$$
which implies the claim for $U$ and $V$ (see also the proof of Lemma 5 in \cite{Koltchinskii_Lounici_HS} for more details 
on this argument). 

Consequently, the distribution of $n \,\Xi_r$ coincides with the distribution of 
\begin{align*}
\Lambda_r:=\left(\begin{array}{c}
\vspace{0.1cm}  2\sqrt{2}\bigl\langle V, \tilde V \bigr\rangle\\
\vspace{0.1cm}  2\sqrt{2}\bigl\langle \tilde V, \bar V \bigr\rangle\\
\vspace{0.1cm}  2(\| V\|_2^2 - \mathbb E \| V\|_2^2 )\\
 2( \| \tilde V\|_2^2 - \mathbb E \| \tilde V\|_2^2 )\\ 
 2( \| \bar V\|_2^2 - \mathbb E \| \bar V\|_2^2 )
\end{array}\right).
\end{align*}

Note that 
$$
\langle V,\tilde V\rangle = \sum_{k,l\in \Delta_r} \sqrt{\gamma_k}\sqrt{\tilde \gamma_l}
\langle \phi_k \otimes C_r X^{(k)}, \tilde \phi_l \otimes C_r \tilde X^{(l)}\rangle
=\mu_r\sum_{k,l\in \Delta_r} \langle \phi_k, \tilde \phi_l \rangle \langle C_r X^{(k)}, C_r \tilde X^{(l)}\rangle
+\eta,
$$
where 
$$
\eta:=\sum_{k,l\in \Delta_r} (\sqrt{\gamma_k}\sqrt{\tilde \gamma_l}-\mu_r)\langle \phi_k, \tilde \phi_l \rangle \langle C_r X^{(k)}, C_r \tilde X^{(l)}\rangle
$$
For the remainder $\eta,$ the following bound holds:
$$
|\eta|\leq \biggl(\sum_{k,l\in \Delta_r}(\sqrt{\gamma_k}\sqrt{\tilde \gamma_l}-\mu_r)^2\biggr)^{1/2}
\biggl(\sum_{k\in \Delta_r}\|C_r X^{(k)}\|^2 \sum_{l\in \Delta_r}\|C_r \tilde X^{(l)}\|^2\biggr)^{1/2},
$$
which, using the independence of $\gamma_k, \tilde \gamma_l, C_r X^{(k)}, C_r \tilde X^{(l)}$ easily implies 
that 
$$
{\mathbb E}|\eta| 
\leq \biggl({\mathbb E}\sum_{k,l\in \Delta_r}(\sqrt{\gamma_k}\sqrt{\tilde \gamma_l}-\mu_r)^2\biggr)^{1/2}
\biggl({\mathbb E}\sum_{k\in \Delta_r}\|C_r X^{(k)}\|^2 {\mathbb E}\sum_{l\in \Delta_r}\|C_r \tilde X^{(l)}\|^2\biggr)^{1/2}
$$
$$
\leq m_r \biggl({\mathbb E}\sum_{k,l\in \Delta_r}(\sqrt{\gamma_k}\sqrt{\tilde \gamma_l}-\mu_r)^2\biggr)^{1/2}
{\mathbb E}\|C_r X\|^2.
$$
Observe also that 
$$
\Bigl|\sqrt{\gamma_k}\sqrt{\tilde \gamma_l}-\mu_r\Bigr|
\leq \frac{\gamma_k \tilde \gamma_l-\mu_r^2}{\mu_r}
\leq |\gamma_k-\mu_r|+ |\tilde \gamma_l-\mu_r|+ \frac{|\gamma_k-\mu_r||\tilde \gamma_l-\mu_r|}{\mu_r},
$$
which implies 
$$
\sum_{k,l\in \Delta_r}(\sqrt{\gamma_k}\sqrt{\tilde \gamma_l}-\mu_r)^2
\leq 
3 m_r \sum_{k\in \Delta_r}(\gamma_k-\mu_r)^2
+
3 m_r \sum_{l\in \Delta_r}(\tilde \gamma_l-\mu_r)^2
+
\frac{3}{\mu_r^2} \sum_{k\in \Delta_r}(\gamma_k-\mu_r)^2
\sum_{l\in \Delta_r}(\tilde \gamma_l-\mu_r)^2
$$
$$
\leq 3 m_r \|\Gamma_r-\mu_r P_r\|_2^2 +3 m_r  \|\tilde \Gamma_r-\mu_r P_r\|_2^2
+ \frac{3}{\mu_r^2} \|\Gamma_r-\mu_r P_r\|_2^2\|\tilde \Gamma_r-\mu_r P_r\|_2^2.
$$
Hence, we get (using independence of $\Gamma_r, \tilde \Gamma_r$)
$$
{\mathbb E}\sum_{k,l\in \Delta_r}(\sqrt{\gamma_k}\sqrt{\tilde \gamma_l}-\mu_r)^2
\leq 
3 m_r {\mathbb E}\|\Gamma_r-\mu_r P_r\|_2^2 +3 m_r  {\mathbb E}\|\tilde \Gamma_r-\mu_r P_r\|_2^2
+ \frac{3}{\mu_r^2} {\mathbb E}\|\Gamma_r-\mu_r P_r\|_2^2{\mathbb E}\|\tilde \Gamma_r-\mu_r P_r\|_2^2.
$$
Since $\Gamma_r, \tilde \Gamma_r$ are sample covariances based on $n$ i.i.d. centered 
Gaussian observations with the true covariance $\mu_r P_r,$
we easily get 
$$
{\mathbb E}\|\Gamma_r-\mu_r P_r\|_2^2 = {\mathbb E}\|\tilde \Gamma_r-\mu_r P_r\|_2^2
\leq \frac{{\mathbb E}\|P_r X\|^4}{n}
\lesssim \frac{\Bigl({\rm tr}(\mu_r P_r)\Bigr)^2}{n}= \frac{\mu_r^2 m_r^2}{n}.
$$
Therefore, 
$$
{\mathbb E}\sum_{k,l\in \Delta_r}(\sqrt{\gamma_k}\sqrt{\tilde \gamma_l}-\mu_r)^2
\lesssim  \frac{\mu_r^2 m_r^3}{n} + \frac{\mu_r^2 m_r^4}{n^2}. 
$$
This yields the following bound on ${\mathbb E}|\eta|:$
\begin{align}
\label{bd_et}
&
\nonumber
{\mathbb E}|\eta|\lesssim \biggl(\frac{\mu_r m_r^{5/2}}{\sqrt{n}} + \frac{\mu_r m_r^3}{n}\biggr) 
{\mathbb E}\|C_r X\|^2 = \biggl(\frac{\mu_r m_r^{5/2}}{\sqrt{n}} + \frac{\mu_r m_r^3}{n}\biggr)
{\rm tr}(C_r \Sigma C_r)
\\
&
\lesssim \frac{\|\Sigma\|_{\infty}^2}{\bar g_r^2}\biggl(\frac{m_r^{5/2}}{\sqrt{n}} \vee \frac{m_r^3}{n}\biggr) 
{\bf r}(\Sigma). 
\end{align}

Similarly, we have
\begin{align*}
\|V\|_2^2 &= \sum_{k\in \Delta_r} \gamma_k \|\phi_k \otimes C_r X^{(k)}\|_2^2=\sum_{k\in \Delta_r} \gamma_k \|C_r X^{(k)}\|^2,
\end{align*}
which implies 
$$
\|V\|_2^2 - {\mathbb E}\|V\|_2^2 = \mu_r\sum_{k\in \Delta_r} \Bigl[\|C_r X^{(k)}\|^2-{\mathbb E}\|C_r X^{(k)}\|^2\Bigr]+\zeta,
$$
where 
$$
\zeta := \sum_{k\in \Delta_r} \Bigl[(\gamma_k-\mu_r)\|C_r X^{(k)}\|^2-{\mathbb E}(\gamma_k-\mu_r)\|C_r X^{(k)}\|^2\Bigr].
$$
The following bound is immediate
$$
{\mathbb E}|\zeta|\leq 2 {\mathbb E}\max_{k\in \Delta_r}|\gamma_k-\mu_r| \sum_{k\in \Delta_r}{\mathbb E}\|C_r X^{(k)}\|^2
\leq 2m_r {\mathbb E}\|\Gamma_r-\mu_r P_r\|_{\infty}{\mathbb E}\|C_r X\|^2,
$$
where we used the independence of random variables $\gamma_k, k\in \Delta_r$ and $C_r X^{(k)}, k\in \Delta_r.$
Applying the bound of Theorem \ref{th_operator} to the sample covariance $\Gamma_r,$ we easily get 
$$
{\mathbb E}\|\Gamma_r-\mu_r P_r\|_{\infty}
\lesssim \mu_r \biggl(\sqrt{\frac{m_r}{n}}\vee \frac{m_r}{n}\biggr).
$$
Therefore, we can conclude that 
\begin{equation}
\label{bd_z}
{\mathbb E}|\zeta|
\lesssim m_r \mu_r \biggl(\sqrt{\frac{m_r}{n}}\vee \frac{m_r}{n}\biggr){\rm tr}(C_r\Sigma C_r)
\lesssim \frac{\|\Sigma\|_{\infty}^2}{\bar g_r^2}\biggl(\frac{m_r^{3/2}}{\sqrt{n}}\vee \frac{m_r^2}{n}\biggr){\bf r}(\Sigma).
\end{equation}

As a consequence of (\ref{bd_et}), (\ref{bd_z}) and similar bounds for other components of vector $\Lambda_r,$ we get that  
\begin{align}
\label{vec-vec}
\left(\begin{array}{c}
\vspace{0.1cm} 2 \sqrt{2}\bigl\langle V, \tilde V \bigr\rangle\\
\vspace{0.1cm} 2 \sqrt{2}\bigl\langle \tilde V, \bar V \bigr\rangle\\
\vspace{0.1cm} 2( \| V\|_2^2 - \mathbb E \| V\|_2^2) \\
 2( \| \tilde V\|_2^2 - \mathbb E \|\tilde  V\|_2^2) \\
 2( \| \bar V\|_2^2 - \mathbb E \|\bar  V\|_2^2) 
\end{array}\right)&= 2\mu_r  \tilde \Theta_r + \xi,
\end{align}
where
$$
\tilde \Theta_r= \left(\begin{array}{c}
\vspace{0.1cm} \sqrt{2}  \Bigl \langle   \sum_{k\in \Delta_r} \phi_k \otimes C_r X^{(k)} , \sum_{k\in \Delta_r}\tilde \phi_{k} \otimes C_r \tilde X^{(k)} \Bigr \rangle\\
\vspace{0.1cm} \sqrt{2}  \Bigl \langle   \sum_{k\in \Delta_r}\tilde \phi_{k} \otimes C_r \tilde X^{(k)} , \sum_{k\in \Delta_r}\bar \phi_{k} \otimes C_r \bar X^{(k)} \Bigr \rangle\\
\vspace{0.1cm} \sum_{k\in \Delta_r} \|C_r X^{(k)} \|^2 -  \sum_{k\in \Delta_r} \mathbb E\|C_r X^{(k)} \|^2 \\
\sum_{k\in \Delta_r}\|C_r \tilde X^{(k)}\|^2 - \sum_{k\in \Delta_r}\mathbb E \|C_r \tilde X^{(k)}\|^2 \\
\sum_{k\in \Delta_r}\|C_r \bar X^{(k)}\|^2 - \sum_{k\in \Delta_r}\mathbb E \|C_r \bar X^{(k)}\|_2^2 
\end{array}\right)
$$
and $\xi\in {\mathbb R}^5$ is a random vector with the components satisfying the following 
bound:
$$
\max_{1\leq j\le 5}{\mathbb E}|\xi_j|\lesssim \frac{\|\Sigma\|_{\infty}^2}{\bar g_r^2}\biggl(\frac{m_r^{5/2}}{\sqrt{n}} \vee \frac{m_r^3}{n}\biggr) 
{\bf r}(\Sigma).
$$
%This immediately follows from (\ref{bd_et}), (\ref{bd_z}) and similar bounds for other components of the vector 
%in the left hand side of (\ref{vec-vec}).
%
It remains to show that the distribution of $\tilde \Theta_r$ coincides with the distribution of $\Theta_r.$
To this end, note that the following representation holds: 
$$
C_r X^{(k)}= \sum_{s\neq r} \frac{1}{\mu_r-\mu_s} P_s X^{(k)}
= \sum_{s\neq r}\frac{\mu_s^{1/2}}{\mu_r-\mu_s} \sum_{j\in \Delta_s}\eta_{j,k} \theta_j,\ k\in \Delta_r, 
$$
where, for all $s\geq 1,$ $\theta_j, j\in \Delta_s$ is an orthonormal basis of the eigenspace of $\Sigma$
corresponding to the eigenvalue $\mu_s$ and $\{\eta_{j,k}\}$ are i.i.d. standard normal random 
variables.  Similarly, we have 
$$
C_r \tilde X^{(k)}
= \sum_{s\neq r}\frac{\mu_s^{1/2}}{\mu_r-\mu_s} \sum_{j\in \Delta_s}\tilde \eta_{j,k} \theta_j,\ k\in \Delta_r, 
$$
and 
$$
C_r \bar X^{(k)}
= \sum_{s\neq r}\frac{\mu_s^{1/2}}{\mu_r-\mu_s} \sum_{j\in \Delta_s}\bar \eta_{j,k} \theta_j,\ k\in \Delta_r, 
$$
where $\{\tilde \eta_{j,k}\}, \{\bar \eta_{j,k}\}$ are i.i.d. standard normal random variables (also independent 
of $\{\eta_{j,k}\}$). Moreover, in addition $\{\eta_{j,k}\}, \{\tilde \eta_{j,k}\}, \{\bar \eta_{j,k}\}$ are independent 
of the samples $X_1,\dots, X_n, \tilde X_1,\dots, \tilde X_n, \bar X_1,\dots, \bar X_n.$ 
Denote 
$$
\eta_j := \sum_{k\in \Delta_r}\eta_{j,k}\theta_k, \tilde \eta_j:= \sum_{k\in \Delta_r}\tilde \eta_{j,k}\theta_k,
\bar \eta_j := \sum_{k\in \Delta_r} \bar \eta_{j,k}\theta_k.
$$
We will also need 
$$
\eta_j' := \sum_{k\in \Delta_r}\eta_{j,k}\phi_k, \tilde \eta_j':= \sum_{k\in \Delta_r}\tilde \eta_{j,k}\tilde \phi_k,
\bar \eta_j' := \sum_{k\in \Delta_r} \bar \eta_{j,k}\bar \phi_k.
$$
Note that conditionally on $\phi_k, \tilde \phi_k, \bar \phi_k,$ the distributions of random vectors 
$\eta_j', \tilde \eta_j', \bar \eta_j', j\in \Delta_s, s\neq r$ is the same as the distribution of random vectors 
$\eta_j, \tilde \eta_j, \bar \eta_j, j\in \Delta_s, s\neq r$
(that are independent ``standard normal" random vectors in the eigenspace of the eigenvalue $\mu_r$).  
In addition to this,
$$
\|\eta_j\|^2=\|\eta_j'\|^2, \|\tilde \eta_j\|^2=\|\tilde \eta_j'\|^2, \|\bar \eta_j\|^2=\|\bar \eta_j'\|^2.
$$
By a straightforward computation, the vector $\tilde \Theta_r$ can be written as follows:
$$
\tilde \Theta_r = \left(\begin{array}{c}
\sqrt{2}\sum_{s\neq r} \frac{\mu_s}{(\mu_r - \mu_s)^2} \sum_{j \in \Delta_s} \langle \eta_j', \tilde \eta_j'\rangle\\
\sqrt{2}\sum_{s\neq r} \frac{\mu_s}{(\mu_r - \mu_s)^2} \sum_{j \in \Delta_s} \langle \tilde \eta_j', \bar \eta_j'\rangle \\
\sum_{s\neq r} \frac{\mu_s}{(\mu_r - \mu_s)^2} \sum_{j \in \Delta_s} [\|\eta_j\|^2 - {\mathbb E}\|\eta_j\|^2]\\
\sum_{s\neq r} \frac{\mu_s}{(\mu_r - \mu_s)^2} \sum_{j \in \Delta_s} [\|\tilde \eta_j\|^2 - {\mathbb E}\|\tilde \eta_j\|^2]\\
\sum_{s\neq r} \frac{\mu_s}{(\mu_r - \mu_s)^2} \sum_{j \in \Delta_s} [\|\bar \eta_j\|^2 - {\mathbb E}\|\bar \eta_j\|^2]
\end{array}\right),
$$
and it has the same distribution as
$$ 
\left(\begin{array}{c}
\sqrt{2}\sum_{s\neq r} \frac{\mu_s}{(\mu_r - \mu_s)^2} \sum_{j \in \Delta_s} \langle \eta_j, \tilde \eta_j\rangle\\
\sqrt{2}\sum_{s\neq r} \frac{\mu_s}{(\mu_r - \mu_s)^2} \sum_{j \in \Delta_s} \langle \tilde \eta_j, \bar \eta_j\rangle \\
\sum_{s\neq r} \frac{\mu_s}{(\mu_r - \mu_s)^2} \sum_{j \in \Delta_s} [\|\eta_j\|^2 - {\mathbb E}\|\eta_j\|^2]\\
\sum_{s\neq r} \frac{\mu_s}{(\mu_r - \mu_s)^2} \sum_{j \in \Delta_s} [\|\tilde \eta_j\|^2 - {\mathbb E}\|\tilde \eta_j\|^2]\\
\sum_{s\neq r} \frac{\mu_s}{(\mu_r - \mu_s)^2} \sum_{j \in \Delta_s} [\|\bar \eta_j\|^2 - {\mathbb E}\|\bar \eta_j\|^2]
\end{array}\right)=\Theta_r.
$$
This completes the proof of the lemma.
\qed
\end{proof}

\section{Proofs: limit theorems}

In this section, we turn to the proofs of theorems \ref{th_odin}, \ref{th_dva} and \ref{th_tri}. Recall the asymptotic framework 
of Section 3 in which $X_1^{(n)}, \dots, X_n^{(n)},$ $\tilde X_1^{(n)}, \dots, \tilde X_n^{(n)}$ and $\bar X_1^{(n)}, \dots, \bar X_n^{(n)}$
are three samples of size $n$ each consisting of i.i.d. copies of a centered Gaussian random vector $X^{(n)}$ with covariance 
$\Sigma^{(n)}.$ Similarly to the non-asymptotic framework, we consider the spectral decomposition $\Sigma^{(n)}=\sum_{r\geq 1}\mu_r^{(n)}P_r^{(n)}$ and we are interested in the estimation of the spectral projector $P^{(n)}=P_{r_n}^{(n)}$ of $\Sigma^{(n)}$
corresponding to its eigenvalue $\mu^{(n)}=\mu_{r_n}^{(n)}$ of multiplicity $m^{(n)}=m_{r_n}^{(n)}.$ 
We define three sample covariance operators (based on the three samples of size $n$): 
$$
\hat \Sigma^{(n)}:=  \frac{1}{n}\sum_{i=1}^n X_i^{(n)} \otimes X_i^{(n)},\quad  \tilde \Sigma^{(n)}:=  \frac{1}{n}\sum_{i=1}^n \tilde X_i^{(n)} \otimes\tilde X_i^{(n)},\quad \bar \Sigma^{(n)}:=  \frac{1}{n}\sum_{i=1}^n \bar X_i^{(n)} \otimes \bar X_i^{(n)}
$$
and set 
$$
E^{(n)}:=\hat \Sigma^{(n)} - \Sigma^{(n)},\quad \tilde E^{(n)} := \tilde \Sigma^{(n)} - \Sigma^{(n)},\quad \bar E^{(n)} := \bar \Sigma^{(n)} - \Sigma^{(n)}.
$$ 
Recall that $C^{(n)}=C_{r_n}^{(n)}= \sum_{s\neq r_n}\frac{1}{\mu_{r_n}^{(n)}-\mu_s^{(n)}}P_s^{(n)}$ and 
$$
B_n=B_{r_n}(\Sigma^{(n)})= 2\sqrt{2}\|C^{(n)}\Sigma^{(n)}C^{(n)}\|_2\|P^{(n)}\Sigma^{(n)}P^{(n)}\|_2.
$$
For a bounded linear operator 
$W:{\mathbb H}\mapsto {\mathbb H},$ 
we will denote, 
$$
L^{(n)}(W)=L_{r_n}^{(n)}(W):=P^{(n)}WC^{(n)}+ C^{(n)}WP^{(n)}.
$$
Recall that, in theorems \ref{th_odin}, \ref{th_dva} and \ref{th_tri}, it is supposed that Assumption \ref{ass_sigma_n} is satisfied
and, moreover, that $\mu^{(n)}$ is the eigenvalue of multiplicity $m^{(n)}=1.$ 
In this case, $\Delta_{r_n}^{(n)}=\{k_n\}$ for some $k_n\geq 1.$

Define the following sequences of random vectors with values in $\R^5$:
\begin{align*}
\Xi^{(n)} := \left(\begin{array}{c}
\sqrt{2}\bigl \langle  L^{(n)}(E^{(n)}) , L^{(n)}(\tilde E^{(n)})  \bigr \rangle \vspace{.1cm}
 \\
 \sqrt{2}\bigl \langle  L^{(n)}(\tilde E^{(n)}) , L^{(n)}(\bar E^{(n)})  \bigr \rangle \vspace{.1cm}
 \\
\|L^{(n)}(E^{(n)})\|_2^2 - \mathbb E \|L^{(n)}(E^{(n)})\|_2^2 
\vspace{.1cm}
\\
\|L^{(n)}(\tilde E^{(n)})\|_2^2 - \mathbb E \|L^{(n)}(\tilde E^{(n)})\|_2^2 
\vspace{.1cm}
\\
\|L^{(n)}(\bar E^{(n)})\|_2^2 - \mathbb E \|L^{(n)}(\bar E^{(n)})\|_2^2
\end{array}
\right)
\end{align*}
and 
\begin{align*}
\Theta^{(n)} :=  \left(\begin{array}{c}
\sqrt{2}\sum_{s\neq r_{n}} \frac{\mu_{s}^{(n)}}{(\mu_{r_{n}}^{(n)} - \mu_{s}^{(n)})^2} \sum_{j \in \Delta_{s}} \eta_{j,k_n}^{(n)}\tilde \eta_{j,k_n}^{(n)}  \vspace{.1cm}
 \\
 \sqrt{2}\sum_{s\neq r_{n}} \frac{\mu_{s}^{(n)}}{(\mu_{r_{n}}^{(n)} - \mu_{s}^{(n)})^2} \sum_{j \in \Delta_{s}} \tilde \eta_{j,k_n}^{(n)}\bar \eta_{j,k_n}^{(n)}  \vspace{.1cm}
 \\
\sum_{s\neq r_{n}} \frac{\mu_{s}^{(n)}}{(\mu_{r_{n}}^{(n)} - \mu_{s}^{(n)})^2} \sum_{j \in \Delta_{s}}[(\eta_{j,k_n}^{(n)})^2-1]\vspace{.1cm}
\\
\sum_{s\neq r_{n}} \frac{\mu_{s}^{(n)}}{(\mu_{r_{n}}^{(n)} - \mu_{s}^{(n)})^2} \sum_{j \in \Delta_{s}}[(\tilde \eta_{j,k_n}^{(n)})^2-1]\\
\sum_{s\neq r_{n}} \frac{\mu_{s}^{(n)}}{(\mu_{r_{n}}^{(n)} - \mu_{s}^{(n)})^2} \sum_{j \in \Delta_{s}}[(\bar 
\eta_{j,k_n}^{(n)})^2-1]
\end{array}
\right),
\end{align*}
where $\eta_{j,k}, \tilde \eta_{j,k}, \bar \eta_{j,k}, j,k\geq 1$ are i.i.d. standard normal random variables. 
Denote 
$$
\bar B_n :=\biggl(2\sum_{s\neq r_{n}} \frac{m_s^{(n)}(\mu_{s}^{(n)})^2}{(\mu_{r_{n}}^{(n)} - \mu_{s}^{(n)})^4}\biggr)^{1/2}.  
$$
It is immediate to see that $B_n= 2\mu^{(n)}\bar B_n$ and, in view of Lemma \ref{lemma:joint-technical-2}, 
$$
n\Xi^{(n)} = 2\mu^{(n)}\tilde \Theta^{(n)}+ \xi^{(n)},
$$
where $\tilde \Theta^{(n)}$ has the same distribution as $\Theta^{(n)}$
and the remainder $\xi^{(n)}\in {\mathbb R}^5$ satisfies the bound
$$
\max_{1\leq j\leq 5}{\mathbb E}|\xi_j^{(n)}|\lesssim 
\biggl(\frac{\|\Sigma^{(n)}\|_{\infty}}{\bar g^{(n)}}\biggr)^2\frac{{\bf r}(\Sigma^{(n)})}{\sqrt{n}}.
$$
(where we also used the assumption that $m^{(n)}=1$). Under Assumption \ref{ass_sigma_n}, 
this implies that 
$$
\frac{\xi^{(n)}}{B_n} =o_{{\mathbb P}}(1)\ {\rm as}\ n\to\infty,
$$
and we get 
\begin{equation}
\label{Xi-asympt}
\frac{n \Xi^{(n)}}{B_n} = \frac{\tilde \Theta^{(n)}}{\bar B_n} + o_{{\mathbb P}}(1).
\end{equation}

We need a simple lemma that will allow us to prove 
that the sequence of random variables $\frac{\tilde \Theta^{(n)}}{\bar B_n}$ is asymptotically 
standard normal implying the same limit distribution for $\frac{n \Xi^{(n)}}{B_n}.$ 

Let $\{\eta, \eta_k^{(n)},\tilde \eta_k^{(n)},\bar \eta_k^{(n)}, \,k\geq 1\}$ be i.i.d. standard normal random variables
and let $\lambda_k^{(n)}>0, k\geq 1, n\geq 1$ be positive real numbers with $\sum_{k\geq 1}\lambda_k^{(n)}<\infty, n\geq 1.$
Define
$$
\vartheta_n :=  \left(\begin{array}{c}
\sqrt{2}\sum_{k \geq 1} \lambda_k^{(n)} \eta_{k}^{(n)}\tilde \eta_{k}^{(n)}  \vspace{.1cm}
 \\
 \sqrt{2}\sum_{k \geq 1} \lambda_k^{(n)} \tilde \eta_{k}^{(n)}\bar \eta_{k}^{(n)}  \vspace{.1cm}
 \\
\sum_{k \geq 1} \lambda_k^{(n)}[(\eta_{k}^{(n)})^2-1]\vspace{.1cm}
\\
\sum_{k \geq 1} \lambda_k^{(n)}[(\tilde \eta_{k}^{(n)})^2-1]\\
\sum_{k \geq 1} \lambda_k^{(n)}[(\bar \eta_{k}^{(n)})^2-1]
\end{array}
\right)
$$
%We suppose again that $\bar A_{n} = \sum_{k\geq 1} \lambda_k^{(n)} < \infty$ and 
and let
$$
\bar B_{n} = \biggl(2\sum_{k\geq 1} ( \lambda_k^{(n)})^2\biggr)^{1/2},\ n\geq 1.
$$

\begin{lemma}\label{joint-clt-theta}
If 
$$
\frac{\bar B_{n}}{\sup_{k\geq 1} \lambda_k^{(n)}} \rightarrow \infty,\quad n\rightarrow \infty,
$$
then the sequence of random vectors
$$
\frac{1}{\bar B_{n}} \vartheta_{n},\quad n\geq 1
$$
converges in distribution to a standard normal random vector $Z_5$ in ${\mathbb R}^5.$
\end{lemma}

\begin{proof}
The proof of this result is an easy application of Lindeberg version of the CLT.
We will establish the convergence in distribution of $\langle \vartheta_n, a \rangle$ to a normal random variable $N(0,|a|^2)$ for an
arbitrary $a\in \R^5.$
%We now check the Lindeberg condition.
For a vector $a=(a_1,\dots, a_5)\in \R^5,$ set 
$$
\vartheta_n(a,k):=
a_1 \sqrt{2} \eta_{k}^{(n)}\tilde \eta_{k}^{(n)} +  a_2 \sqrt{2}\tilde \eta_{k}^{(n)}\bar \eta_{k}^{(n)}  + a_3 [(\eta_{k}^{(n)})^2-1] + a_4 [(\tilde \eta_{k}^{(n)})^2-1] + a_5 [(\bar \eta_{k}^{(n)})^2-1],\ k\geq 1.
$$
Without loss of generality, assume that $|a|=1.$
Note that r.v. $\vartheta_n(a,k), k\geq 1$ are i.i.d., ${\mathbb E}\vartheta_n(a,k)=0$ 
and ${\rm Var}(\vartheta_n(a,k))=2.$
Therefore, for 
$$
\zeta_n(a):= \frac{1}{\bar B_{n}}\langle \vartheta_n, a \rangle=\frac{\sum_{k\geq 1}\lambda_k^{(n)} \vartheta_n(a,k) }{\bar B_n},
$$ 
it holds that  
${\mathbb E}\zeta_n(a)=0$ and ${\rm Var}(\zeta_n(a))=1.$ 
In textbook versions of the central limit theorem, the 
result is usually stated for sums of finite triangular arrays 
of independent random variables. In our case, the sums are infinite.
However, it is easy to reduce the problem to the finite case by truncating
the series to $p_n$ terms, where $p_n$ is such that $\sum_{k>p_n}\lambda_{k}^{(n)}=o(\bar B_n).$ 
Such a reduction is rather simple and will be skipped. 
%Note that 
%$$
%\frac{\sup_{k\geq 1}(\lambda_{k}^{(n)})^2 
%{\mathbb E}\Bigl[(\xi_k^{(n)})^2-1\Bigr]^2}{\bar B_n^2}=
%{\mathbb E}(\xi^2-1)^2\frac{\sup_{k\geq 1} (\lambda_{k}^{(n)})^2}{\bar B_n^2}
%\to 0,
%$$
%and
%$$
%\frac{\sup_{k\geq 1}(\lambda_{k}^{(n)})^2 
%{\mathbb E}\Bigl[\eta_{k}^{(n)}\xi_{k}^{(n)}\Bigr]^2}{\bar B_n^2}=
%{\mathbb E}[\eta^2\xi^2]\frac{\sup_{k\geq 1} (\lambda_{k}^{(n)})^2}{\bar B_n^2}
%\to 0.
%$$
%Note also that $\{\vartheta_n(a,k)\}_{k\geq 1}$ is i.i.d. Therefore, for any fixed $a\in \R^5$, we get
By the assumption of the lemma,  
$$
\frac{\sup_{k\geq 1}(\lambda_{k}^{(n)})^2 
{\mathbb E}\Bigl[\vartheta_n^2(a,k)\Bigr]}{\bar B_n^2}=
\frac{2\sup_{k\geq 1} (\lambda_{k}^{(n)})^2}{\bar B_n^2}
\to 0.
$$
It remains to check that the Lindeberg 
condition holds. To this end, note that 
$$
|\vartheta_n(a,k)| \leq 
\max\Bigl(\sqrt{2} |\eta_{k}^{(n)}||\tilde \eta_{k}^{(n)}|, \sqrt{2}|\tilde \eta_{k}^{(n)}||\bar \eta_{k}^{(n)}|, |(\eta_{k}^{(n)})^2-1|, |\tilde \eta_{k}^{(n)})^2-1|, |(\bar \eta_{k}^{(n)})^2-1|\Bigr)
$$
and observe that the random variables involved in the maximum in the right hand side 
are sub-exponential. This easily implies the following bound on the tails of $\vartheta_n(a,k)$
$$
{\mathbb P}\{|\vartheta_n(a,k)| \geq t\}\leq 5 e^{-ct}, t\geq 0
$$ 
that holds with some numerical constant $c>0$ and for all $a\in {\mathbb R}^5, |a|=1$ and all 
$k\geq 1.$ This bound also implies that ${\mathbb E}|\vartheta_n(a,k)|^4\leq C,$ $a\in {\mathbb R}^5, 
\|a\|=1$
for some numerical constant $C>0.$
Therefore, for all $\tau>0,$ we have
\begin{align*}
&\frac{\sum_{k\geq 1} (\lambda_k^{(n)})^2 
{\mathbb E}\left[ \vartheta_n^2(a,k) I\biggl(\lambda_k^{(n)} |\vartheta_n(a,k)|\geq \tau 
\bar B_n\biggr)\right]}{\bar B_n^2}\\
&\hspace{0.5cm} \leq \frac{1}{\bar B_n^2}\sum_{k\geq 1} \left( \lambda_k^{(n)} \right)^2 \mathbb{E}^{1/2}|\vartheta_n(a,k)|^4 \mathbb{P}^{1/2}\left( \lambda_k^{(n)} |\vartheta_n(a,k)|\geq \tau 
\bar B_n  \right)\\
&\hspace{0.5cm} 
\lesssim \frac{\sum_{k\geq 1}\Bigl(\lambda_k^{(n)}\Bigr)^2}{\bar B_n^2}
\exp\biggl\{-\frac{c\tau \bar B_n}{2\sup_{k\geq 1} \lambda_k^{(n)}}\biggr\}
\lesssim \exp\biggl\{-\frac{c\tau \bar B_n}{2\sup_{k\geq 1} \lambda_k^{(n)}}\biggr\},
%\leq \frac{1}{\bar B_n^2} \mathbb{E}^{1/2}\left[ \vartheta_n^4(a,1)\right]   \sum_{k\geq 1} \left( \lambda_k^{(n)} \right)^2  \mathbb{P}%^{1/2}\left( \lambda_k^{(n)} |\vartheta_n(a,1)|\geq \tau 
%\bar B_n  \right).
\end{align*}
which tends to $0$ as $n\to\infty$ (under the condition that $\frac{\bar B_n}{\sup_{k\geq 1}\lambda_k^{(n)}}\to \infty$).
\qed 
\end{proof}

Lemma \ref{joint-clt-theta} will be applied to the sequence of random vectors $\Theta^{(n)}.$
Under Assumption \ref{ass_sigma_n}, the condition of the lemma holds since 
$$
\frac{1}{\bar B_n}\sup_{s\neq r_n} \frac{\mu_s^{(n)}}{(\mu_{r_n}^{(n)}-\mu_s^{(n)})^2}
=\frac{2}{B_n} \sup_{s\neq r_n} \frac{\mu^{(n)}\mu_s^{(n)}}{(\mu_{r_n}^{(n)}-\mu_s^{(n)})^2}
\leq \frac{2}{B_n}\biggl(\frac{\|\Sigma^{(n)}\|_{\infty}}{\bar g^{(n)}}\biggr)^2 \to 0\ {\rm as}\ n\to\infty.
$$
Thus, Lemma \ref{joint-clt-theta} implies that $\frac{\Theta^{(n)}}{\bar B_n}\stackrel{d}{\longrightarrow} Z_5$ and, in view of (\ref{Xi-asympt}),
we also have that 
\begin{equation}
\label{ass-normal}
\frac{n\Xi^{(n)}}{B_n}\stackrel{d}{\longrightarrow} Z_5\ {\rm as}\ n\to\infty. 
\end{equation}

Under Assumption \ref{ass_sigma_n}, Lemma \ref{HS-conc-bd} easily implies that 
\begin{align}\label{interm-final1}
\frac{n}{B_n}\left(\|\hat P^{(n)} - P^{(n)}\|_2^2 - \mathbb E\|\hat P^{(n)} - P^{(n)}\|_2^2\right) = 
\frac{n}{B_n}\langle \Xi^{(n)},u\rangle + o_{\mathbb P}(1),\; u = (0,0,1,0,0).
\end{align}
Under the same assumption, Lemma \ref{lemma2} implies that  
\begin{align}\label{interm-final2}
\frac{n}{B_n}\left((1+\hat b^{(n)})^2  - (1+ b^{(n)})^2 \right) =  \frac{n}{B_n}\langle \Xi^{(n)},v\rangle  + o_{\mathbb P}(1),\; 
v = \left(\frac{1}{\sqrt{2}},0,-\frac{1}{2},-\frac{1}{2},0\right)
\end{align}
and
\begin{align}
\label{interm-final2'}
\frac{n}{B_n}\left( (1+\tilde b^{(n)})^2  - (1+ b^{(n)})^2 \right)= \frac{n}{B_n}\langle \Xi^{(n)},w\rangle  + o_{\mathbb P}(1),\; w= \left(0,\frac{1}{\sqrt{2}},0,-\frac{1}{2},-\frac{1}{2}\right).
\end{align}
It follows from the last two relationships that 
\begin{align}
\label{interm-final2_2'}
\frac{n}{B_n}\left((1+\hat b^{(n)})^2  - (1+ \tilde b^{(n)})^2 \right) =  
\frac{n}{B_n}\langle \Xi^{(n)},v-w\rangle  + o_{\mathbb P}(1). 
\end{align}

{\bf Proof of Theorem \ref{th_odin}}.
Note that
\begin{align}
\label{bias-obs}
\frac{n}{B_{n}}(\hat b^{(n)} -b^{(n)} ) = \frac{n}{B_{n}} \frac{(1+ \hat b^{(n)})^2  -(1+ b^{(n)})^2}{2+ \hat b^{(n)} +  b^{(n)}}.
\end{align}
Under Assumption  \ref{ass_sigma_n}, Proposition \ref{prop:tildetheta} implies that $|\hat b^{(n)} - b^{(n)}| = O_{\mathbb P} \left(\frac{\sqrt{\mathbf{r}(\Sigma^{(n)})}}{n}\right)$. Recall also that 
$|b^{(n)}| \lesssim \frac{\|\Sigma^{(n)}\|_{\infty}^2}{\bar g_r^2} \frac{\mathbf{r}(\Sigma^{(n)})}{n}$
(see bound (\ref{b_r_bd})).
Thus, under Assumption \ref{ass_sigma_n}, we get that $b^{(n)} = o(1)$ and $\hat b^{(n)} = o_{\mathbb{P}}(1).$ 
These facts along with 
representations (\ref{bias-obs}), (\ref{interm-final2}) and also with (\ref{ass-normal}) imply that 
$\frac{2n}{B_n}(\hat b^{(n)}-b^{(n)})$ converges in distribution 
to the same limit as $\frac{n}{B_n}\left((1+\hat b^{(n)})^2  - (1+ b^{(n)})^2 \right),$
which is the distribution of the random variable $\langle Z_5,w\rangle.$ Since $|w|=1,$ $\langle Z_5,w\rangle$
is a standard normal random variable, which completes the proof of Theorem \ref{th_odin}.  

{\bf Proof of Theorem \ref{th_tri}}. 
Recall that
\begin{align*}
\E \| \hat P^{(n)} - P^{(n)} \|_2^2 = -2 b^{(n)}
\end{align*}
(see (\ref{risk-bias})). 
The following representation holds:
\begin{align}
\label{data-driven-interm1}
&\frac{\|\hat P^{(n)} - P^{(n)} \|_2^2 + 2\hat b^{(n)} }{|(1+\hat b^{(n)})^2  - (1+ \tilde b^{(n)})^2|}\notag\\
&\hspace{1cm}= \frac{\|\hat P^{(n)} - P^{(n)} \|_2^2 + 2 b^{(n)} }{|(1+\hat b^{(n)})^2  - (1+ \tilde b^{(n)})^2|}  + \frac{2(\hat b^{(n)} - b^{(n)}) }{|(1+\hat b^{(n)})^2  - (1+ \tilde b^{(n)})^2|} \notag\\
 &\hspace{1cm}= \frac{\frac{n}{B_n}\left(\|\hat P^{(n)} - P^{(n)} \|_2^2  - \E \| \hat P^{(n)} - P^{(n)} \|_2^2\right) }{\left|\frac{n}{B_n}\left((1+\hat b^{(n)})^2  - (1+ \tilde b^{(n)})^2\right)\right|}  + \frac{\frac{2n}{B_n}\left(\hat b^{(n)} - b^{(n)}\right) }{\left|\frac{n}{B_n}\left((1+\hat b^{(n)})^2  - (1+ \tilde b^{(n)})^2\right)\right|}.
% &\hspace{1cm}= \frac{\|\hat P^{(n)} - P^{(n)} \|_2^2 - \E \| \hat P^{(n)} - P^{(n)} \|_2^2 }{\sqrt{\mathrm{Var}\left(\| \hat P^{(n)} - P^{(n)} \|_2^2\right)\left(1 +  o_{\mathbb P}\left( 1\right) \right)}  }  + \frac{2(\hat b^{(n)} - b^{(n)}) }{\sqrt{\mathrm{Var}\left(\| \hat P^{(n)} - P^{(n)} \|_2^2\right)\left(1 +  o_{\mathbb P}\left( 1\right)\right) }}\\
% &\hspace{1cm}= \left(\frac{\|\hat P^{(n)} - P^{(n)} \|_2^2 - \E \| \hat P^{(n)} - P^{(n)} \|_2^2 }{\sqrt{\mathrm{Var}\left(\| \hat P^{(n)} - P^{(n)} \|_2^2\right)}  }  + \frac{2n(\hat b^{(n)} - b^{(n)}) }{B_n}\right) \left(1 +  o_{\mathbb P}\left( 1\right) \right)
\end{align}
In view of (\ref{ass-normal}), (\ref{interm-final1}), (\ref{interm-final2_2'}) and the combination of 
(\ref{bias-obs}) with (\ref{interm-final2}),
we easily conclude that the sequence of random variables 
$$
\frac{\|\hat P^{(n)} - P^{(n)} \|_2^2 + 2\hat b^{(n)} }{|(1+\hat b^{(n)})^2  - (1+ \tilde b^{(n)})^2|}
$$ 
converges in distribution to
$
\frac{\langle Z_5,u+v\rangle}{|\langle Z_5,v-w\rangle|}.
$
Using Proposition \ref{Cauchy}, it is easy to show that $\frac{\langle Z_5,u+v\rangle}{|\langle Z_5,v-w\rangle|} \stackrel{d}{=} Y_{\frac{5}{6},\frac{\sqrt{47}}{6}}.$
This completes the proof of Theorem \ref{th_tri}.

{\bf Proof of Theorem \ref{th_dva}} is quite similar. 

\qed

\bibliographystyle{plain}
\bibliography{biblio_A}

\end{document}